\documentclass[11pt]{amsart}

\usepackage{a4wide, amsmath, amsfonts, amssymb, mathrsfs, amsthm, bbm, appendix}
\usepackage[hyperfootnotes=false]{hyperref}

\usepackage{xargs}
\usepackage[pdftex,dvipsnames]{xcolor}  

\usepackage[shortlabels]{enumitem}

\usepackage{xcolor}
\numberwithin{equation}{section}
\usepackage{enumitem}

\newtheorem{theorem}{Theorem}[section]
\newtheorem{lemma}[theorem]{Lemma}
\newtheorem{corollary}[theorem]{Corollary}
\newtheorem{remark}[theorem]{Remark}
\newtheorem{proposition}[theorem]{Proposition}
\newtheorem{definition}[theorem]{Definition}

\newtheorem{assumption}[theorem]{Assumption}

\allowdisplaybreaks[2]

\renewcommand{\epsilon}{\varepsilon}
\newcommand{\R}{\mathbb{R}}
\newcommand{\N}{\mathbb{N}}
\renewcommand{\P}{\mathbb{P}}
\newcommand{\E}{\mathbb{E}}
\newcommand{\indicator}[1]{\mathbbm{1}_{#1}}
\newcommand{\Id}{\,\mathrm{d}}
\newcommand{\norm}[1]{\left\lVert#1\right\rVert}
\newcommand{\qv}[1]{\langle #1 \rangle}
\newcommand{\testfunctions}[1]{C_c^{\infty}(#1)}

\newcommand{\cF}{\mathcal{F}}
\newcommand{\cG}{\mathcal{G}}
\newcommand{\cP}{\mathcal{P}}
\def\L{\mathscr L}

\title[Relative entropy in the common noise setting]{Quantitative relative entropy estimates for interacting particle systems with common noise}

\author[Nikolaev]{Paul Nikolaev}
\address{Paul Nikolaev, University of Mannheim, Germany}
\email{paul.nikolae@uni-mannheim.de}

\date{\today}

\begin{document}

\begin{abstract}
We derive quantitative estimates proving the conditional propagation of chaos for large stochastic systems of interacting particles subject to both idiosyncratic and common noise. We obtain explicit bounds on the relative entropy between the conditional Liouville equation and the stochastic Fokker--Planck equation with an interaction kernel \(k\in L^2(\R^d) \cap L^\infty(\R^d)\), extending far beyond the Lipschitz case. Our method relies on reducing the problem to the idiosyncratic setting, which allows us to utilize the exponential law of large numbers by Jabin and Wang~\cite{JabinWang2018} in a pathwise manner.  
\end{abstract}

\maketitle



\tableofcontents
\section{Introduction}
We prove a conditional propagation of chaos result for the interacting particle system      
\begin{equation} \label{eq: intro_interacting_envir}
\Id X_t^{i} = - \frac{1}{N}\sum\limits_{j=1}^N k(X_t^{i}-X_t^{j}) \Id t + \sigma(t,X_t^{i}) \Id B_t^{i} + \nu(t,X_t^{i}) \Id W_t 
\end{equation}
driven by idiosyncratic noise \((B_t^{i}, t \ge 0)\), \(i \in \N\) and common noise \( (W_t, t \ge 0)\), both represented by multi-dimensional Brownian motions. The Brownian motions \((B_t^{i}, t \ge 0)\) are independent of each other, and \( (W_t, t\ge 0)\) is independent of \((B_t^{i}, t \ge 0)\) for all \(i\). Details of the probabilistic setting are provided in Section~\ref{sec: setting}. For the interaction kernel, we require \(k\in L^2(\R^d) \cap L^\infty(\R^d)\). 
Systems of the form~\eqref{eq: intro_interacting_envir} are commonly utilized in the field of mean-field games~\cite[Section 2.1]{CarmonaReneDElarue2018} as well as mathematical finance~\cite{Ahuja2016,Lacker2020,Hammersley2020,Shkolnikov2023}.

Our aim is to investigate the asymptotic behavior of the system~\eqref{eq: intro_interacting_envir} and to derive macroscopic models, a process known as propagation of chaos~\cite{snitzman_propagation_of_chaos,JabinWang2017}. In this article, we develop the relative entropy method in the common noise setting~\eqref{eq: intro_interacting_envir}. While relative entropy theory is well-established in the absence of common noise~\cite{JabinWang2018,BreschDidierJabinWangZhenfu2019}, there is relatively little literature addressing relative entropy in the case with common noise. The techniques that relative entropy theory relies on—namely, the Liouville equation, the regularity of the mean-field partial differential equation (PDE), and the exponential law of large numbers—need to be adapted for the common noise setting.
Recently, Lacker~\cite{Lacker2023} manage to improve the convergence rate of the relative entropy under stricter assumptions by utilizing the BBGKY hierarchy. 
Previously, Jabin, Bresch, and Soler~\cite{bresch2022} used the BBGKY for VlasovPoisson-Fokker-Planck system for plasmas in dimension two. 

In our setting, the presence of common noise adds an extra layer of complexity to the problem. The empirical distribution of players at the limit now evolves stochastically according to a nonlinear stochastic partial differential equation (SPDE). Moreover, the associated Liouville equation must be modified to a conditional Liouville equation, which also solves an SPDE.

\subsection{Our contribution}
Our first contribution lies in proving the global well-posedness of the stochastic partial differential equations associated with the particle systems. Since our kernel \(k\) lies in \(L^2(\R^d) \cap L^\infty(\R^d)\) it prevents us from directly applying established results in the literature on SPDEs, such as those found in textbooks by Krylov~\cite{Krylov1999AnAA}, Rozovsky~\cite{RozovskyBorisL2018SES} to our nonlinear stochastic Fokker--Planck equation~\eqref{eq: chaotic_spde}. Consequently, for our nonlinear SPDE~\eqref{eq: chaotic_spde}, we resort to employing a Picard iteration method.

Our second contribution focuses on demonstrating the boundedness of the relative entropy in \(\R^{dN}\). To achieve this, we compute the evolution of the relative entropy using  It{\^o}'s formula. A key insight lies in recognizing that, unlike in the classical setting where distributions are directly compared in the relative entropy, we should instead compare the conditional distribution of the interacting particle system~\eqref{eq: intro_interacting_envir} with the solution of the stochastic Fokker--Planck equation~\eqref{eq: chaotic_spde}.  

However, defining the relative entropy for general random measures poses a challenge. To the best of our knowledge, we could not find a reliable definition.  
But if we consider solutions of the conditional Liouville
equation~\eqref{eq: liouville_SPDE} and the stochastic Fokker--Planck equation~\eqref{eq: chaotic_spde} with sufficient regularity, we can provide a pointwise interpretation of the SDPE's~\eqref{eq: liouville_SPDE},~\eqref{eq: chaotic_spde}. This enables us to apply It{\^o}'s formula to the SPDE's, thereby deriving an expression for the dynamics of the relative entropy. Subsequently, we can further analyze this expression by employing methods known in the non-common case, such as the exponential law of large numbers~\cite{JabinWang2018}. Notice, that in the absence of transport-type noise, we must address the quadratic variations arising in the calculations. 
Consequently, we establish the boundedness of the relative entropy for systems with smooth coefficients. 
By applying the subadditivity and the Csiszar--Kullback--Pinsker inequality, we obtain an estimate in the \(L^1\)-norm between the marginals of the Liouville equation and the stochastic Fokker--Planck equation, showing a decay rate of \(N^{-1/2}\), indicating that results from Jabin and Wang~\cite{JabinWang2018} are applicable even in the presence of common noise.  Additionally, we recover the result from Jabin and Wang~\cite{JabinWang2018} on the whole space in the case of vanishing common noise \(\nu=0\). 

Finally, we demonstrate that in the case \(k \in L^2(\R^d) \cap L^\infty(\R^d)\) the conditional Liouville equation~\eqref{eq: liouville_SPDE} and the stochastic Fokker--Planck equation~\eqref{eq: chaotic_spde} can be approximated by the associated SPDE's with smooth coefficients. This type of stability result leads to the relative entropy estimate, which seems to be novel. 
Additionally, we demonstrate conditional propagation of chaos. To our best knowledge, this is the first result on conditional propagation of chaos for general bounded interaction kernels in the \(L^1\)-norm. However, it's crucial to exercise caution at this stage, as the standard equivalent characterization provided by Sznitman~\cite[Proposition 2.2]{snitzman_propagation_of_chaos} cannot be directly applied due to the non-deterministic nature of our limiting measure. Therefore, we need to replicate the proof in the setting of random measures. 
Consequently, we extend the results for conditional propagation of chaos by Carmona and Delarue~\cite{CarmonaReneDElarue2018} to the kernel \(k \in L^2(\R^d) \cap L^\infty(\R^d)\).

\subsection{Related literatures}
In contrast to interacting particle system, which are driven by idiosyncratic noise~\cite{snitzman_propagation_of_chaos,JabinWang2018,RosenzweigSerfaty2023,galeati2024}, literature on common noise remains limited.
For systems with uniformly Lipschitz interaction forces, Coghi and Flandoli~\cite{CoghiFlandoli2016} established conditional propagation of chaos in the presence of common noise. Their approach relied on sharp estimates in Kolmogorov’s continuity theorem and properties of the measure-valued solutions of the associated stochastic Fokker–Planck equation.
Dawson and Vaillancourt~\cite{Dawson1995} also formulated a martingale problem and demonstrated tightness of the empirical measure obtaining a qualitative result with no convergence rates. 

In a parallel effort to ours, Shao and Zao~\cite{Shao2024} recently presented a similar relative entropy estimate for the stochastic two-dimensional Navier--Stokes equation driven by transport noise, utilizing methods from~\cite{JabinWang2018}. Their approach necessitates the initial independence of intensity and position, which introduces additional assumptions on the physical model. While their work shares similarities with ours, such as the relative entropy estimate and utilization of common noise, distinctions arise in the domain (torus vs \(\R^d\)) and the nature of noise (transport noise vs. It{\^o} noise). Moreover, the authors~~\cite{Shao2024} capitalized on the specific structure of the Biot--Savart kernel, particularly its divergence-free property. 

Rosenzweig~\cite{Rosenzweig2020} addressed a related problem by establishing quantitative estimates for the Biot--Savart kernel using the modulated energy method in a pathwise setting, a method previously introduced by Serfaty~\cite{Serfaty_2020} in the deterministic setting for Coulomb kernels. Challenges in this context stemmed from commutator estimate for the second-order correction term. Although the modulated energy method naturally extends to systems with transport noise, the pathwise extension in the case of relative entropy requires a pathwise Radon--Nikodym derivative, posing challenges for random measures.
Furthermore, Rosenzweig, Nguyen, and Serfaty~\cite{Nguyen2022} obtained similar results for the repulsive Coulomb case under transport noise.

In our work, the presence of idiosyncratic noise, dictated by the ellipticity of the diffusion coefficient \(\sigma\), plays a significant role. 
This aspect differentiates our approach from the previous literature discussed. 
In out setting, the work by Huang, Qiu~\cite{HuangQui2021} investigates the Keller--Segel model with Bessel potential and the work of Chen, Prömel and the author~\cite{nikolaev2023hk} demonstrates the existence and uniqueness of nonlinear SPDEs and conditional McKean--Vlasov equations.
Other relevant literature includes the work of Kurtz, Xiong~\cite{KURTZ1999}, Coghi and Gess~\cite{Coghi2019} for the existence of stochastic non-linear Fokker--Planck equations under non-linaer Lipschitz coefficients. 

A broad area where interacting particle systems of the form~\eqref{eq: intro_interacting_envir} are studied is in mathematical finance and mean-field games. For instance, Carmona, Delarue, and Lacker~\cite{CarmonaReneDelarueLacker2016} investigate mean-field games for Lipschitz drifts with common noise of the form~\eqref{eq: intro_interacting_envir}, as do Delarue, Lacker, and Ramanan~\cite{Lacker2020}. We refer to the references therein for more details.

Finally, let us mention the results of Jabin and Wang~\cite{JabinWang2018}, which serve as a foundational reference in our work. We extend their results in the bounded kernel setting to include common noise. Additionally, we require fewer assumptions on the regularity of the mean-field equation, which is a consequence of the stronger assumption on the interaction kernel \(k\) and the ellipticity of the diffusion coefficient \(\sigma\). 
However, the focus on the entire space \(\R^d\) prevents the inclusion of kernels in \(W^{-1,\infty}(\R^d)\), as stability of stochastic partial differential equations cannot be guaranteed using our methods.

\smallskip
\noindent\textbf{Organization of the paper:}
In Section~\ref{sec: setting} we provide the definitions of the particle systems and the associated SPDE's along with the introduction of the relative entropy. In
Section~\ref{sec: existence_ass_spde} the well-posedness of the Liouville equation~\eqref{eq: liouville_SPDE} and the stochastic Fokker--Planck equation~\eqref{eq: chaotic_spde} is established. 
In Section~\ref{sec: relative_entropy_smooth_setting} we compute the evolution of the relative entropy in the setting of smooth coefficients. Finally, 
conditional propagation of chaos with common noise for interaction kernel \(k \in L^2(\R^d) \cap L^\infty(\R^d)\) is investigated in Section~\ref{sec: stability_spde} as well as a stability estimate for the conditional Liouville equation and stochastic Fokker--Planck equation.

\section*{Acknowledgment}
The majority of the work was completed during Paul Nikolaev's research visit to Beijing under the supervision of Prof. Zhenfu Wang. We would like to thank Peking University or their hospitality and Zhenfu Wang for the fruitful discussions regarding relative entropy and the propagation of chaos.
Additionally, we are grateful to Xing Huang for providing useful insights about relative entropy and conditional laws, as well as for contributing to the result in the appendix, which emerged as a byproduct of discussions at a Peking University seminar. Finally, we extend our thanks to Yun Gong for the extensive discussions on modulated energy, relative entropy, and the topic of mean-field limits. His explanations of complex mathematical concepts and his willingness to assist with matters outside the campus greatly contributed to the productivity of our research.

\section{Setting} \label{sec: setting}
We write a vector in \(\R^{dN}\) as \(x = (x_1, \ldots, x_N) \in \R^{dN} \), where \(x_i=(x_{i,1},\ldots,x_{i,d}) \in \R^d\). For a normal vector in \(\R^d\) we will use the variable \(z \in \R^d\). 
For a matrix \(A \in \R^{d \times d'}\) we denote the \((\alpha,\beta)\) entry as \([A]_{(\alpha,\beta)}\)
Throughout the entire paper, we use the generic constant \( C \) for inequalities, which may change from line to line.

\subsection{Probability setting}
In this subsection we introduce the probabilistic setting, in particular, the \(N\)-particle system and the associated McKean--Vlasov equation. To that end, let \((\Omega, \mathcal{F}, ( \mathcal{F}_t)_{t \ge 0 } , \P)\) be a complete probability space with right-continuous filtration \((\mathcal{F}_t)_{t \ge 0 } \). Furthermore, let \((B_t^{i}=(B_t^{i,1},\ldots,B_t^{i,m}), t\ge 0)\), \( i=1, \ldots, N\) be independent \(m\)-dimensional Brownian motions with respect to \((\cF_t, t \ge 0)\) and \((W_t=(W_t^{1},\ldots,W_t^{\tilde{m}}), t\ge 0)\) be another \(\tilde{m}\)-dimensional Brownian motions with respect to \((\cF_t, t \ge 0)\), which is independent from \((B_t^{i}, t\ge 0)\), \( i=1, \ldots, N\). Moreover, we denote by \(\cF^W = (\cF^W_t ,t \ge 0)\) 
the augmented filtration generated by \(W\) and by
\(\cP^W\) the predictable \(\sigma\)-algebra with respect to \(\cF^W\). For the initial data we consider a sequence \((\zeta^{i}, i \in \N)\) of independent \(d\)-dimensional \(\cF_0\)-measurable random variables with density \(\rho_0\), which are independent of the Brownian motions \((B_t^{i}, t\ge 0)\), \( i=1, \ldots, N\) and the filtration \(\cF^W = (\cF^W_t ,t \ge 0)\). 
We require also the following coefficients 
\begin{equation*}
    k \colon \R^d \mapsto \R^d, \quad 
    \sigma \colon [0,T] \times \R^d \mapsto \R^{d \times m}, \quad 
    \nu \colon [0,T] \times \R^d \mapsto \R^{d \times \Tilde{m}}. 
\end{equation*}
We define the interacting particle system by  
\begin{equation} \label{eq: particle_system}
\Id X_t^{i} = - \frac{1}{N}\sum\limits_{j=1}^N k(X_t^{i}-X_t^{j}) \Id t + \sigma(t,X_t^{i}) \Id B_t^{i} + \nu(t,X_t^{i}) \Id W_t 
\end{equation}
and the conditional McKean--Vlasov system by
\begin{align} \label{eq: conditional_mckean_vlasov}
\begin{cases}
\Id Y_t^{i} = - (k*\rho_t(Y_t^{i})) \Id t + \sigma(t,Y_t^{i}) \Id B_t^{i} + \nu(t,Y_t^{i}) \Id W_t  \\
\rho_t = \L_{Y_t^{i} \vert \cF^W_t}, 
\end{cases}
\end{align}
where \(\L_{Y_t^{i} \vert \cF^W_t}\) is the conditional density of \(Y^{i}_t\) given \(\cF^W_t\).
It is well-known that pathwise-uniqueness implies that \( \L_{Y_t^{i} \vert \cF^W_t} \) is independent of \(i \in \N\).

\subsection{Function spaces and basic definitions}

In this subsection we collect some basic definitions and introduce the required function spaces for our SPDE's. For \( 1 \le p \le \infty\) we denote by \(L^p(\R^d)\) with norm \(\norm{\cdot}_{L^p(\R^d)}\) the vector space of measurable functions whose \(p\)-th power is Lebesgue integrable (with the standard modification for \(p = \infty\)), by \(\testfunctions{\R^d}\) the set of all infinitely differentiable functions with compact support on \(\R^d\) and by \(\mathcal{S}(\R^d)\) the set of all Schwartz functions, see \cite[Chapter~6]{YoshidaKosaku1995FA} for more details. We endow \(\testfunctions{\R^d}\) and \(\mathcal{S}(\R^d)\) with their standard topologies. Let
\begin{equation*}
  \mathcal{A}:= \{ \alpha = ( \alpha_1,\ldots, \alpha_d) \; : \; \alpha_1, \ldots, \alpha_d \in \N_{0} \}
\end{equation*}
be the set of all multi-indices and \(|\alpha| := \alpha_1 +  \ldots+ \alpha_d\). The derivative will be denoted by
\begin{equation*}
 \frac{\partial^{|\alpha|}}{\partial x_1^{\alpha_1} \partial x_2^{\alpha_2} \cdots \partial x_d^{\alpha_d} } \quad \mathrm{or} \quad \partial x_1^{\alpha_1} \partial x_2^{\alpha_2} \cdots \partial x_d^{\alpha_d}.
\end{equation*}
 We drop the superscripts \(\alpha_i\) in the case \(|\alpha_i| = 1 \). For each \(\alpha, \beta \in \mathcal{A} \) we define the semi-norms
\begin{align*}
  p_{\alpha, \beta} (f):= \sup\limits_{z \in \R^d} | z^\alpha( \partial^\beta f)(z)|.
\end{align*}
Equipped with these seminorms, \(\mathcal{S}(\R^d)\) is  a Fr{\'e}chet space \cite[Appendix~A.5]{AbelsHelmut2012}. Furthermore, we introduce the space of Schwartz distributions \(\mathcal{S}'(\R^d)\). We denote dual parings by \(\qv{\cdot, \cdot}\). For instance, for \(f \in \mathcal{S}', \; u \in \mathcal{S}\) we have \(\qv{u,f}  = u[f]\) and for a probability measure~\(\mu\) we have \(\qv{f,\mu}  = \int f \Id \mu\). The correct interpretation will be clear from the context but should not be confused with the scalar product \(\qv{\cdot, \cdot}_{L^2(\R^d)} \) in \(L^2(\R^d)\).

The Fourier transform \(\mathcal{F}[u]\) and the inverse Fourier transform \(\mathcal{F}^{-1}[u]\) for $u \in \mathcal{S}'(\R^d)$ and $ f \in \mathcal{S}(\R^d)$ are defined by
\begin{equation*}
  \qv{\mathcal{F}[u],f} := \qv{u, \mathcal{F}[f]},
\end{equation*}
where \(\mathcal{F}[f]\) and \(\mathcal{F}^{-1}[f]\) is given by
\begin{equation*}
  \mathcal{F}[f](\xi) := \frac{1}{(2\pi)^{d/2}} \int e^{-i \xi \cdot z } f(z) \Id z
  \quad \text{and}\quad
  \mathcal{F}^{-1}[f](\xi):= \frac{1}{(2\pi)^{d/2}} \int e^{i \xi \cdot z } f(z) \Id z .
\end{equation*}

For each \(s \in \R\) we denoted the Bessel potential by \(J^s := (1-\Delta)^{s/2}u := \mathcal{F}^{-1}[(1+|\xi|^2)^{s/2} \mathcal{F}[u]] \) for \(u \in \mathcal{S}'(\R^d)\). We define the Bessel potential space \(\mathnormal{H}_p^s \) for \(p \in [1,\infty)\) and \(s \in \R\) by
\begin{equation*}
  \mathnormal{H}_p^s:= \{ u \in \mathcal{S}'(\R^d) \; : \;  (1-\Delta)^{s/2} u \in L^p(\R^d) \}
\end{equation*}
with the norm
\begin{equation*}
  \norm{u}_{\mathnormal{H}^{s}_p}:= \norm{(1-\Delta)^{s/2}u}_{L^p(\R^d)}, \quad u \in  \mathnormal{H}_p^s.
\end{equation*}
For \(1 <p < \infty , \; m \in \N\) we can characterize the above Bessel potential spaces \(\mathnormal{H}_p^m\) as Sobolev spaces
\begin{align*}
  W^{m,p} (\R^d) : = \bigg \{ f \in L^p(\R^d) \; : \; \norm{f}_{W^{m,p}(\R^d)}:= \sum\limits_{ \alpha \in \mathcal{A} ,\  |\alpha|\le m } \norm{\partial^\alpha f }_{L^p(\R^d)}  < \infty \bigg \},
\end{align*}
where \(\partial^\alpha f \) is to be understood as weak derivatives \cite{AdamsRobertA2003Ss}. We refer to \cite[Theorem~2.5.6]{TriebelHans1983Tofs} for the proof of the above characterization. As a result, we use Sobolev spaces, which in our context are easier to handle, instead of Bessel potential spaces, whenever possible.

Finally, we introduce general \(L^p\)-spaces. For a Banach space \((E, \norm{\cdot}_{E})\), some filtration \((\cF_t)_{t\geq0}\), \(1 \le p \le \infty\) and \(0 \le s <t \le T\) we denote by \(S^p_{\cF}([s,t];E )\) the set of \(E\)-valued \((\mathcal{F}_t)\)-adapted continuous processes \((X_u, u \in [s,t])\) such that
\begin{align*}
  \norm{X}_{S^p_{\cF}([s,t];E )}:=
  \begin{cases}
  \bigg( \E \bigg( \sup\limits_{u \in [s,t]} \norm{X_u}_E^p\bigg) \bigg)^{\frac{1}{p}}, \quad & p \in [1, \infty)  \\
  \sup\limits_{ \omega \in \Omega } \sup\limits_{u \in [s,t]} \norm{X_u}_E , \quad & p = \infty
  \end{cases}
\end{align*}
is finite. Similar, \(L^p_{\cF}([s,t];E)\) denotes the set of \(E\)-valued predictable processes \((X_u, u \in [s,t])\) such that
\begin{align*}
  \norm{X}_{L^p_{\cF}([s,t];E )}:=
  \begin{cases}
  \bigg( \E \bigg( \int\limits_{s}^t  \norm{X_u}_E^p \Id u \bigg) \bigg)^{\frac{1}{p}}, \quad & p \in [1, \infty)  \\
  \sup\limits_{ (\omega,u)  \in \Omega \times [s,t]} \norm{X_u}_E , \quad & p = \infty
  \end{cases}
\end{align*}
is finite. In most case \(Z\) will be the Bessel potential space \(\mathnormal{H}_p^n\), as it is mainly used by Krylov~\cite{krylov2010ito} in treating SPDEs. For a more detail introduction to the above function spaces we refer to \cite[Section~3]{Krylov1999AnAA}.

At the end, let us introduce the concept of relative entropy for probability measures.   
For two probability measures \(\mu_1, \mu_2\) over \(E\), we define the relative entropy \(\mathcal{H}(\mu_1 \vert \mu_2)\) as 
\begin{equation*}
\mathcal{H}(\mu_1 \vert \mu_2)
= 
\begin{cases}
\int_E \frac{\mathrm{d} \mu_1 }{\mathrm{d} \mu_2 } \log\bigg( \frac{ \mathrm{d}\mu_1}{\mathrm{d} \mu_2 }\bigg) \Id \mu_2 ,  \quad \quad &\mathrm{if} \; \mu_1 \ll \mu_2, \\
\infty, &\mathrm{otheriwse}. 
\end{cases}
\end{equation*}
Let us recall two crucial properties of the relative entropy. 
For simplicity, let us assume that for \(r\in \N\) we have two probability densities \(f^r\) and \(g^r\) on \(E^r\). Then, the Csiszar--Kullback--Pinsker inequality~\cite[Chapter~22]{Villani2009}
\begin{equation} \label{eq: CKP_inequality}
    \norm{f^m-g^m}_{L^1(E^r)}^2 \le 2 \mathcal{H}(f^r \vert g^r) 
\end{equation}
holds. Additionally, if \(f^r\) is a symmetric probability measure for all \(r \le N\) and \(g^r\) is a tensor product \(g^{\otimes r}\) for some probability density \(g\), then the subadditivy inequality~\cite[Proposition 19]{Wang2017} 
\begin{equation}
    \mathcal{H}(f^r \vert g^r)  \le \frac{r}{N} \mathcal{H}(f^N \vert g^N) 
\end{equation}
holds for all \(r \le N \).

\subsection{Stochastic partial differential equations}

Let us introduce the SPDE's associated to the SDE's defined in~\eqref{eq: particle_system},~\eqref{eq: conditional_mckean_vlasov}. 
From the interacting particle system~\eqref{eq: particle_system} we can derive the stochastic Fokker--Planck equation in \(\R^{dN}\) solved by the conditional density \(\rho_t^N\) given the filtration \(\cF_t^W\). We call this SPDE the conditional Liouville equation, which is given by
\begin{align} \label{eq: liouville_SPDE}
\Id \rho_t^N
=& \;  \sum\limits_{i=1}^N \nabla_{x_i} \cdot \bigg( \frac{1}{N} \sum\limits_{j=1}^N k(x_i-x_j) \rho_t^N \bigg) \Id t 
- \sum\limits_{i=1}^N \nabla_{x_i} \cdot (\nu(t,x_{i})  \rho_t^N \Id W_t)  \nonumber \\
&\; + \frac{1}{2} \sum\limits_{i,j=1}^N \sum\limits_{\alpha,\beta=1}^d \partial_{x_{i,\alpha}}\partial_{x_{j,\beta}} \bigg( ( [\sigma(t,x_i)\sigma(t,x_j)^{\mathrm{T}}]_{(\alpha,\beta)} \delta_{i,j} + [\nu(t,x_i)\nu(t,x_j)^{\mathrm{T}}]_{(\alpha,\beta)} ) \rho_t^N  \bigg)  \Id t .
\end{align}

Analogously, we define the
\(d\)-dimensional non-linaer SPDE 
\begin{align} \label{eq: d_dim_spde}
\begin{split}
\Id \rho_t
=& \;  \nabla \cdot (( k*\rho_t) \rho_t ))  \Id t 
- \nabla \cdot (\nu_t  \rho_t \Id W_t)  \\
&\; + \frac{1}{2}  \sum\limits_{\alpha,\beta=1}^d \partial_{z_{\alpha}}\partial_{z_{\beta}} \bigg( ( [\sigma_t\sigma_t^{\mathrm{T}}]_{(\alpha,\beta)} + [\nu_t \nu_t^{\mathrm{T}}]_{(\alpha,\beta)} ) \rho_t \bigg)  \Id t
\end{split}
\end{align}
associated to~\eqref{eq: conditional_mckean_vlasov}. 
The provided equation represents the conditional density of a single particle in the McKean--Vlasov SDE~\eqref{eq: conditional_mckean_vlasov}. Extending this concept to cover \(N\) conditionally independent particles, we introduce the SPDE

\begin{align} \label{eq: chaotic_spde}
\Id \rho_t^{\otimes N }
=& \;  \sum\limits_{i=1}^N \nabla_{x_i} \cdot (( k*\rho_t)(x_i) \rho_t^{\otimes N } ))  \Id t 
- \sum\limits_{i=1}^N \nabla_{x_i} \cdot (\nu(t,x_{i})  \rho_t^{\otimes N }  \Id W_t) \nonumber  \\
& + \frac{1}{2} \sum\limits_{i,j=1}^N \sum\limits_{\alpha,\beta=1}^d \partial_{x_{i,\alpha}}\partial_{x_{j,\beta}} \bigg( ( [\sigma(t,x_i)\sigma(t,x_j)^{\mathrm{T}}]_{(\alpha,\beta)} \delta_{i,j} + [\nu(t,x_i)\nu(t,x_j)^{\mathrm{T}}]_{(\alpha,\beta)} ) \rho_t^{\otimes N }  \bigg)  \Id t. 
\end{align}
The notation \(\otimes N\) is intentionally used, as it will become evident that if the \(d\)-dimensional SPDE~\eqref{eq: d_dim_spde} has a solution, then equation~\eqref{eq: chaotic_spde} also possesses a solution in the form of a tensor product. 

From the above equation we can deduce the following \(r\)-marginal SPDE of the interacting particle system 
\begin{align*}
\Id \rho_t^{r,N} 
=& \;  \sum\limits_{i=1}^r \int_{\R^{(N-r)d}} \nabla_{x_i} \cdot \bigg( \frac{1}{N} \sum\limits_{j=1}^N  k(x_i-x_j) \rho_t^N \bigg) \Id x_{r+1} \cdot \Id x_N  \Id t \\
&- \sum\limits_{i=1}^N \nabla_{x_i} \cdot ( \int_{\R^{(N-r)d}} \nu(t,x_{i}) \rho_t^N \Id x_{r+1} \cdots \Id x_N \Id W_t) \\
&\; + \frac{1}{2} \sum\limits_{i,j=1}^N \sum\limits_{\alpha,\beta=1}^d \partial_{x_{i,\alpha}}\partial_{x_{j,\beta}} \bigg( \int_{\R^{(N-r)d}} ( [\sigma(t,x_i)\sigma(t,x_j)^{\mathrm{T}}]_{(\alpha,\beta)} \delta_{i,j} \\
&+ [\nu(t,x_i)\nu(t,x_j)^{\mathrm{T}}]_{(\alpha,\beta)} ) \rho_t^N  \bigg) \Id x_{r+1} \cdots x_N \Id t 
\end{align*}

We provide now definitions to the SPDE's~\eqref{eq: liouville_SPDE},~\eqref{eq: d_dim_spde} and~\eqref{eq: chaotic_spde}. 
\begin{definition}\label{def: solution_of_liouville_equation}
For a fix \(N\in \N\) a non-negative stochastic process \((\rho_t^N , t \ge 0)\) is called a (weak) solution of the SPDE~\eqref{eq: chaotic_spde} with initial data \(\rho^{\otimes N}_0\) if
  \begin{equation*}
    \rho^N  \in  L^2_{\cF^W}([0,T];H^1(\R^{dN}) ) \cap S^\infty_{\cF^W}([0,T];L^1(\R^{dN}) )
  \end{equation*}
  and, for any \(\varphi \in \testfunctions{\R^{dN}}\), \(\rho^N\) satisfies almost surely the equation, for all \(t \in [0,T]\),
  \begin{align*}
    \begin{split}
    &\; \qv{\rho_t^N, \varphi}_{L^2(\R^{dN})} \\
    =& \;  \qv{\rho_0^{\otimes N}, \varphi}_{L^2(\R^{dN})}
    - \sum\limits_{i=1}^N \int\limits_0^t   \bigg \langle \frac{1}{N} \sum\limits_{j=1}^N k(x_i-x_j) \rho_s^N , \nabla_{x_i} \varphi \bigg \rangle_{L^2(\R^{dN})}  \Id s  \\
    &    + \frac{1}{2} \sum\limits_{i=1}^N \sum\limits_{\alpha,\beta=1}^d  \int\limits_0^t \bigg\langle   [\sigma(s,x_i)\sigma(s,x_i)^{\mathrm{T}}]_{(\alpha,\beta)} \rho_s^N , \partial_{x_{i,\beta}}  \partial_{x_{i,\alpha}} \varphi \bigg\rangle_{L^2(\R^{dN})} \Id s \\
    &+  \frac{1}{2} \sum\limits_{i,j=1}^N \sum\limits_{\alpha,\beta=1}^d  \int\limits_0^t \bigg\langle   [\nu(s,x_i)\nu(s,x_j)^{\mathrm{T}}]_{(\alpha,\beta)} \rho_s^N , \partial_{x_{j,\beta}}  \partial_{x_{i,\alpha}} \varphi \bigg\rangle_{L^2(\R^{dN})} \Id s \\
    &    + \sum\limits_{i=1}^N  \sum\limits_{\alpha=1}^d \sum\limits_{\hat{l}=1}^{\tilde{m}} \int\limits_0^t \bigg\langle    \nu^{\alpha,\hat{l}}(s,x_i) \rho_s^{ N} ,  \partial_{x_{i,\alpha}} \varphi \bigg\rangle_{L^2(\R^{dN})}  W_s^{\hat{l}}. 
    \end{split}
  \end{align*}
\end{definition}

\begin{definition}
    A non-negative stochastic process \((\rho_t,t \ge 0)\) is a solution to the SPDE~\eqref{eq: d_dim_spde} with initial data \(\rho_0\), if 
    \begin{equation}
        \rho \in  L^2_{\cF^W}([0,T];H^1(\R^{d}) ) \cap S^\infty_{\cF^W}([0,T];L^1(\R^{d}) )
    \end{equation}
    and
    \begin{align*} 
    \begin{split}
     \qv{\rho_t, \varphi}_{L^2(\R^{dN})} 
    =&\;   \qv{\rho_0, \varphi}_{L^2(\R^{d})}
    -  \int\limits_0^t   \bigg \langle (k*\rho_s) \rho_s, \nabla_{x_i} \varphi \bigg \rangle_{L^2(\R^{d})}  \Id s  \\
   &    + \frac{1}{2} \sum\limits_{\alpha,\beta=1}^d  \int\limits_0^t \bigg\langle   [\sigma(s,z)\sigma(s,z)^{\mathrm{T}}]_{(\alpha,\beta)} \rho_s, \partial_{z_{\beta}}  \partial_{z_{\alpha}} \varphi \bigg\rangle_{L^2(\R^{d})} \Id s \\
    &+  \frac{1}{2} \sum\limits_{\alpha,\beta=1}^d  \int\limits_0^t \bigg\langle   [\nu(s,z)\nu(s,z)^{\mathrm{T}}]_{(\alpha,\beta)} \rho_s, \partial_{z_{\beta}}  \partial_{z_{\alpha}} \varphi \bigg\rangle_{L^2(\R^{d})} \Id s \\
    &    +  \sum\limits_{\alpha=1}^d \sum\limits_{\hat{l}=1}^{\tilde{m}} \int\limits_0^t \bigg\langle    \nu^{\alpha,\hat{l}}(s,z) \rho_s ,  \partial_{z_{\alpha}} \varphi \bigg\rangle_{L^2(\R^{d})}  W_s^{\hat{l}}. 
    \end{split}
  \end{align*}
  holds. 
\end{definition}

\begin{definition}\label{def: solution_of_chaotic_spde}
For a fix \(N\in \N\) a non-negative stochastic process \((\rho_t^{\otimes N} , t \ge 0)\) is called a (weak) solution of the SPDE~\eqref{eq: liouville_SPDE} with initial data \(\rho_0^{\otimes N}\) if
  \begin{equation*}
    \rho^{\otimes N} \in  L^2_{\cF^W}([0,T];H^1(\R^{dN}) ) \cap S^\infty_{\cF^W}([0,T];L^1(\R^{dN}) )
  \end{equation*}
  and, for any \(\varphi \in \testfunctions{\R^{dN}}\), \(\rho^{\otimes N}\) satisfies almost surely the equation, for all \(t \in [0,T]\),
  \begin{align} \label{eq: def_chaotic_spde}
    \begin{split}
    &\; \qv{\rho_t^{\otimes N}, \varphi}_{L^2(\R^{dN})} \\
    =& \;  \qv{\rho_0^{\otimes N}, \varphi}_{L^2(\R^{dN})}
    - \sum\limits_{i=1}^N \int\limits_0^t   \bigg \langle (k*\rho_s)(x_i) \rho_s^{\otimes{N}}, \nabla_{x_i} \varphi \bigg \rangle_{L^2(\R^{dN})}  \Id s  \\
   &    + \frac{1}{2} \sum\limits_{i=1}^N \sum\limits_{\alpha,\beta=1}^d  \int\limits_0^t \bigg\langle   [\sigma(s,x_i)\sigma(s,x_i)^{\mathrm{T}}]_{(\alpha,\beta)} \rho_s^{\otimes N} , \partial_{x_{i,\beta}}  \partial_{x_{i,\alpha}} \varphi \bigg\rangle_{L^2(\R^{dN})} \Id s \\
    &+  \frac{1}{2} \sum\limits_{i,j=1}^N \sum\limits_{\alpha,\beta=1}^d  \int\limits_0^t \bigg\langle   [\nu(s,x_i)\nu(s,x_j)^{\mathrm{T}}]_{(\alpha,\beta)} \rho_s^{\otimes N}, \partial_{x_{j,\beta}}  \partial_{x_{i,\alpha}} \varphi \bigg\rangle_{L^2(\R^{dN})} \Id s \\
    &    + \sum\limits_{i=1}^N  \sum\limits_{\alpha=1}^d \sum\limits_{\hat{l}=1}^{\tilde{m}} \int\limits_0^t \bigg\langle    \nu^{\alpha,\hat{l}}(s,x_i) \rho_s^{\otimes{N}} ,  \partial_{x_{i,\alpha}} \varphi \bigg\rangle_{L^2(\R^{dN})}  W_s^{\hat{l}} . 
    \end{split}
  \end{align}
\end{definition}

Throughout this article we make the following assumption on our diffusion coefficients \(\sigma,\nu\) and the inital data \(\rho_0\).
\begin{assumption} \label{ass: diffusion_coef}
\begin{enumerate}
\item For each \((i,l) \in \{1,\ldots,d\} \times \{1,\ldots,m\}\) we require \(\sigma^{i,l}(t,\cdot) \in C^1(\R^d)\) and for some positive constant \(C>0\) we have the following uniform estimate \[\sup\limits_{0\le t \le T} \norm{\sigma^{i,l}(t,\cdot)}_{C^1(\R^d)} \le C .\] 
\item \label{item: regularity_common_noise} For each \((i,\hat{l}) \in \{1,\ldots,d\} \times \{1,\ldots,\tilde{m}\}\) we require \(\nu^{i,\hat{l}}(t,\cdot) \in C^1(\R^d)\) and for some positive constant \(C>0\) we have the following uniform estimate \[\sup\limits_{0\le t \le T} \norm{\nu^{i,\hat{l}}(t,\cdot)}_{C^1(\R^d)} \le C .\] 
\item \label{item: divergence_free_common_noise} The diffusion \(\nu\) is divergence free, i.e. for each \(\hat{l}=1,\ldots,\tilde{m}\) we have 
\begin{equation*} 
\sum\limits_{\beta=1}^d \partial_{y_{\beta}} \nu^{\beta,\hat{l}} (t,y) = 0 . 
\end{equation*}
\item \label{item: cancelation_sigma_nu}For every \(\alpha \in \{1,\ldots,d\}\) we have 
\begin{align*} 
\begin{split}
\sum\limits_{\beta=1}^d \partial_{y_{\beta}} \sum\limits_{l=1}^m \sigma^{\alpha,l}(t,z) \sigma^{\beta,l}(t,z) &= 0,, \quad (t,z) \in [0,T] \times \R^d ,  \\
\sum\limits_{\beta=1}^d \partial_{y_{\beta}}  \sum\limits_{\hat{l}=1}^{\tilde{m}} \nu^{\alpha,\hat{l}}(t,z) \nu^{\beta,\hat{l}}(t,z) &= 0 , \quad (t,z) \in [0,T] \times \R^d . 
\end{split}
\end{align*}
\item \label{item: ellipticity_idiosy} \(\sigma\) satisfies the ellipticity condition. For all \(\lambda \in \R^d\) we have 
\begin{equation*} 
\sum\limits_{\alpha,\beta=1}^d      [\sigma_s(z)\sigma_s(z)^{\mathrm{T}}]_{(\alpha,\beta)}    \lambda_{\alpha} \lambda_{\beta} \ge \delta |\lambda|^2 . 
\end{equation*}

\item The initial data \(\rho_0 \in L^2(\R^d) \) satisfies the first moment estimate 
\begin{equation*}
    \int_{\R^d} |z|^2 \rho_0(z) \Id z< \infty. 
\end{equation*}

\end{enumerate}

\end{assumption}

A direct consequence of Assumption~\ref{ass: diffusion_coef}~\eqref{item: ellipticity_idiosy} is the existence and strong uniqueness of the interacting particle system~\eqref{eq: particle_system}. 

\begin{remark}
  Assumption~\eqref{item: divergence_free_common_noise} is a standard assumption in the field of SPDEs~\cite{CoghiFlandoli2016,HuangQui2021}, implying the non-randomness of the \(L^p\)-norms. Condition~\eqref{item: cancelation_sigma_nu} is necessary for reducing the conditional Liouville equation~\eqref{eq: liouville_SPDE} from a Fokker--Planck equation~\cite{BogachevKrylovRöckner2015} to a SPDE in divergence structure. The same applies to the non-local SPDE~\eqref{eq: chaotic_spde}. We require the ellipticity condition~\eqref{item: ellipticity_idiosy} for the existence of the SPDE's~\cite{RozovskyBorisL2018SES,Krylov1999AnAA}, analogously to the parabolic PDE theory.    
\end{remark}

\section{Well-posedness of the stochastic PDE's} \label{sec: existence_ass_spde}
In this section we demonstrate the existence and uniqueness of solutions to the linear SPDE~\eqref{eq: liouville_SPDE} and the non-linear SPDE~\eqref{eq: chaotic_spde} in the sense of Definition~\ref{def: solution_of_liouville_equation} and Definition~\ref{def: solution_of_chaotic_spde}. Additionally, we provide some a priori bounds uniform in the probablity space. 
We use similar techniques to~\cite{nikolaev2023hk} and for readers familiar with SPDE's the results should not come as a surprise.

We start by demonstarting the existence of the Liouville  equation~\eqref{eq: liouville_SPDE}. 
\begin{proposition}
Let \(k\in L^\infty(\R^d)\). Then, there exists a solution \(\rho^N \in L^2_{\cF^W}([0,T];H^1(\R^{d}))\) of the SPDE~\eqref{eq: liouville_SPDE} in the sense of Definition~\ref{def: solution_of_liouville_equation}. 
\end{proposition}
\begin{proof}
Define \(f(t,u) = \sum\limits_{i=1}^N  \nabla_{x_i} \cdot \bigg( \frac{1}{N} \sum\limits_{j=1}^N k(x_i-x_j) u \bigg)\) for \(u \in H^1(\R^{dN})\). Then we find 
\begin{equation*}
\norm{f(t,u)}_{H^{-1}(\R^{dN})} 
\le \sum\limits_{i=1}^N \norm{   \frac{1}{N}\sum\limits_{j=1}^N  k(x_i-x_j) u  }_{L^2(\R^{dN})} 
\le N \norm{k}_{L^\infty(\R^d)} \norm{u}_{L^2(\R^{dN})}
\end{equation*}
and by linearity 
\begin{align*}
\norm{  f(t,u) -f(t,v) }_{H^{-1}(\R^{dN})} 
\le N \norm{k}_{L^\infty(\R^d)} \norm{u-v}_{L^2(\R^{dN})}
\end{align*}
for \(u,v \in H^1(\R^{dN})\). Applying~\cite[Theorem 5.1]{Krylov1999AnAA} proves the proposition. 
\end{proof}

\begin{lemma}[A priori \(L^2\)-estimate] \label{lem: priori_l2}
Let \(k\in L^\infty(\R^d)\) and let the non-negative process \(\rho \in L^2_{\cF^W}([0,T];H^1(\R^{d}))\) satisfy the SPDE 
\begin{align*}
\Id \rho_t
=& \;  \nabla \cdot (( k*u_t) \rho_t ))  \Id t 
- \nabla \cdot (\nu_t  \rho_t \Id W_t)  \\
&\; + \frac{1}{2}  \sum\limits_{\alpha,\beta=1}^d \partial_{z_{\alpha}}\partial_{z_{\beta}} \bigg( ( [\sigma_t\sigma_t^{\mathrm{T}}]_{(\alpha,\beta)} + [\nu_t \nu_t^{\mathrm{T}}]_{(\alpha,\beta)} ) \rho_t \bigg)  \Id t,
\end{align*}
where \(u_t\) is a predictable with respect to \(\cF^W\), \(L^1\)-valued process, which satisfies \(\norm{u_t}_{L^1(\R^d)} \le 1 ,\; \P\text{-a.s.}\) for all \(t \ge 0\). 
Then 
 \begin{equation} \label{eq: uniform_L2_bound}
  \norm{\rho_t}_{L^2(\R^d))}  \le C(T,d,\delta, \norm{k}_{L^{\infty}(\R^d)}) \norm{\rho_0}_{L^{2}(\R^d)}, \quad \P\text{-a.s.} 
 \end{equation}
 for all \(t \ge 0\). 
\end{lemma}
\begin{proof}
Let us compute the norm \(\norm{\rho_t}_{L^2(\R^d)}\),
\begin{align*}
&\; \norm{\rho_t}_{L^2(\R^d)}^2  - \norm{\rho_0}_{L^2(\R^d)}^2  \\
= &\; -2 \int\limits_0^t \int_{\R^d} (( k*u_s) \rho_s) \cdot \nabla \rho_s\Id z \Id s  \\
&-   \sum\limits_{\alpha,\beta=1}^d \int\limits_0^t \int_{\R^d} \partial_{z_{\beta}}  \bigg( ( [\sigma_s\sigma_s^{\mathrm{T}}]_{(\alpha,\beta)} + [\nu_s \nu_s^{\mathrm{T}}]_{(\alpha,\beta)} ) \rho_s \bigg) \partial_{z_{\alpha}} \rho_s \Id z \Id s  \\
&+ \sum\limits_{\hat{l}=1}^{\tilde{m}}\int\limits_0^t \int_{\R^d} \bigg|\sum\limits_{\beta=1}^d \partial_{z_{\beta}} (\nu_s^{\beta,l} \rho_s) \bigg|^2 \Id  z \Id s  +  2  \sum\limits_{\hat{l}=1}^{\tilde{m}} \sum\limits_{\beta=1}^d  \int\limits_0^t  \int_{\R^d} \rho_s \partial_{z_{\beta}} (\nu_s^{\beta,\hat{l}} \rho_s ) \Id W_s^{\hat{l}} . 
\end{align*}  
By the divergence-free assumption of \(\nu\) we have 
\begin{align*}
\int_{\R^d} \rho_s \sum\limits_{\beta=1}^d  \partial_{z_{\beta}} (\nu_s^{\beta,\hat{l}} \rho_s)   \Id  z
&= \frac{1}{2}\int_{\R^d}  \sum\limits_{\beta=1}^d  \nu_s^{\beta,\hat{l}}   \partial_{z_{\beta}} (\rho_s ^2) \Id  z  \\
&=- \frac{1}{2}\int_{\R^d}  \sum\limits_{\beta=1}^d  \partial_{z_{\beta}} \nu_s^{\beta,\hat{l}}   \rho_s^2 \Id  z  \\
&=0 
\end{align*}
and therefore the stochastic integral vanishes. Again using the divergence-free assumption~\ref{ass: diffusion_coef}~\eqref{item: divergence_free_common_noise} we find  
\begin{align} \label{eq: moment_estimate_aux1}
\begin{split}
  \sum\limits_{\hat{l}=1}^{\tilde{m}}\int\limits_0^t \int_{\R^d} \bigg|\sum\limits_{\beta=1}^d \partial_{z_{\beta}} (\nu_s^{\beta,l} \rho_s ) \bigg|^2 \Id z \Id s  
&=  \sum\limits_{\hat{l}=1}^{\tilde{m}} \sum\limits_{\alpha,\beta=1}^d \int\limits_0^t \int_{\R^d} \partial_{z_{\alpha}} (\nu_s^{\alpha,l} \rho_s ) \partial_{z_{\beta}} (\nu_s^{\beta,l} \rho_s )  \Id z \Id s   \\
&=  \sum\limits_{\hat{l}=1}^{\tilde{m}} \sum\limits_{\alpha,\beta=1}^d \int\limits_0^t \int_{\R^d} \nu_s^{\alpha,l} \nu_s^{\beta,l}   \partial_{z_{\alpha}} \rho_s   \partial_{z_{\beta}}\rho_s   \Id  z \Id s  . 
 \end{split}
\end{align}
However, we also have 
 \begin{align*}
&\; \sum\limits_{\alpha,\beta=1}^d \int\limits_0^t \int_{\R^d}  \partial_{z_{\beta}}  \bigg( ( [\sigma_s\sigma_s^{\mathrm{T}}]_{(\alpha,\beta)} + [\nu_s \nu_s^{\mathrm{T}}]_{(\alpha,\beta)} ) \rho_s \bigg) \partial_{z_{\alpha}} \rho_s  \Id z \Id s  \\
= &\;   \sum\limits_{\alpha,\beta=1}^d  \int\limits_0^t \int_{\R^d} \partial_{z_{\beta}}  \bigg(  [\sigma_s\sigma_s^{\mathrm{T}}]_{(\alpha,\beta)} + [\nu_s \nu_s^{\mathrm{T}}]_{(\alpha,\beta)} \bigg) \rho_s  \partial_{z_{\alpha}} \rho_s  \Id z \Id s \\
&+  \sum\limits_{\alpha,\beta=1}^d  \int\limits_0^t \int_{\R^d} ( [\sigma_s\sigma_s^{\mathrm{T}}]_{(\alpha,\beta)} + [\nu_s \nu_s^{\mathrm{T}}]_{(\alpha,\beta)}) \partial_{z_{\beta}}  \rho_s \partial_{z_{\alpha}} \rho_s  \Id z \Id s .  
 \end{align*}
Consequently, the last term containing \(\nu\) is exactly the term \eqref{eq: moment_estimate_aux1} and therefore both terms cancel. For the first term we use the uniform bound on the derivatives and Young's inequality to estimated it by 
\begin{align*}
&\; \sum\limits_{\alpha,\beta=1}^d \bigg| \int\limits_0^t \int_{\R^d} \partial_{z_{\beta}}  \bigg(  [\sigma_s\sigma_s^{\mathrm{T}}]_{(\alpha,\beta)} + [\nu_s \nu_s^{\mathrm{T}}]_{(\alpha,\beta)} \bigg) \rho_s  \partial_{z_{\alpha}} \rho_s  \Id z \Id s\bigg| \\
&\;\le  C(d,\delta) \int\limits_0^t \norm{\rho_s}_{L^2(\R^d)}^2 + \frac{\delta}{2} \norm{\nabla\rho_s}_{L^2(\R^d)}^2 \Id s.
\end{align*}
 Combining the last estimates and plugging them in the evolution of the \(L^2\)-norm, we arrive at 
\begin{align*}
 \norm{\rho_t}_{L^2(\R^d)}^2  -\norm{\rho_0}_{L^2(\R^d)}^2
=&\;  -2 \int\limits_0^t \int_{\R^d} (( k*u_s) \rho_s) \cdot \nabla  \rho_s \Id z \Id s  + C(d,\delta)\norm{\rho_s}_{L^2(\R^d)}^2 \\
&- \sum\limits_{\alpha,\beta=1}^d  \int\limits_0^t \int_{\R^d}  [\sigma_s\sigma_s^{\mathrm{T}}]_{(\alpha,\beta)}  \partial_{z_{\beta}}  \rho_s \partial_{z_{\alpha}} \rho_s \Id z  + \frac{\delta}{2} \norm{\nabla \rho_s}_{L^2(\R^d)}^2 \Id s \\
\le &\;  2 \int\limits_0^t \norm{ k*u_s}_{L^\infty(\R^d)}^2 \norm{\rho_s}_{L^2(\R^d)}^2  \Id s   + C(d,\delta)\norm{\rho_s}_{L^2(\R^d)}^2 \Id s  \\
\le &\;   (2\norm{k}_{L^\infty(\R^d)}^2+ C(d,\delta)) \int\limits_0^t  \norm{\rho_s}^2_{L^2(\R^d)} \Id s , 
\end{align*}
 where we used the ellipticity condition in the last step. An application of Gronwall's lemma provides 
 \begin{equation*}
  \norm{\rho_t}_{L^2(\R^d)}^2  \le C(T,d,\delta, \norm{k}_{L^{\infty}(\R^d)}) \norm{\rho_0}^2_{L^2(\R^d)}, \quad \P\text{-a.s.} 
 \end{equation*}
 for all \(t \ge 0\).
\end{proof}

\begin{lemma}(A priori moment estimate) \label{lem: moment_estimate_spde}
Suppose we are in the setting of Lemma~\ref{lem: priori_l2}.  Then the following moment estimate holds
\begin{equation*}
\E\bigg ( \sup\limits_{0\le t \le T}  \bigg(\int_{\R^d} \rho_t(z) |z|^2 \Id z \bigg)^2 \bigg) 
 \le  C(T,\sigma,\nu,\norm{k}_{L^\infty(\R^d)})  \bigg(\int_{\R^d} \rho_0(x) |z|^2 \Id z\bigg)^2 . 
 \end{equation*} 
\end{lemma}

\begin{proof}
 The core idea is to use \(|z|^2\) as a test function. To that end, we take a sequence of radial non-negative anti-symmetric smooth functions \((g_n, n \in \N)\) with \(g_n \in C_c^2(\R^d)\) for all \( n \in \N \), such that \(g_n\) grows to \(|z|^2\) as \(n \to \infty\) and \(|\nabla g_n|^2 \le C g_n\) and \(|\Delta g_n| \) is uniformly bounded in \(n \in \N\).
For instance one can choose 
  \begin{align*}
    \chi_{n}(z) :=
    \begin{cases}
    |z| \quad   &|z| \ge \frac{1}{n}  \\
    -n^3\frac{|z|^4}{8}  + n \frac{3|z|^2}{4} + \frac{3}{8n} \quad & |z| \le \frac{1}{n}
    \end{cases}
  \end{align*}
and let \((\zeta_n, n \in \N)\) be a sequence of compactly supported cut-off function defined by \(\zeta_n(x) = \zeta(x/n)\), where \(\zeta\) is a smooth function with support in the ball of radius two and has value one in the unit ball.
The reader can veriy that \(g_n= \chi_n^2 \zeta_n\) satisfies the above properties.
  Plugging \(g_n\) as our test function, we obtain 
  \begin{align*}
&\; \qv{\rho_t, g_n}_{L^2(\R^d)} \\
=&\;  \qv{\rho_0, g_n}_{L^2(\R^d)}
    - \int\limits_0^t   \bigg \langle (k*u_s) \rho_s, \nabla g_n \bigg \rangle_{L^2(\R^{d})}  \Id s      +  \sum\limits_{\alpha=1}^d \sum\limits_{\hat{l}=1}^{\tilde{m}} \int\limits_0^t \bigg\langle    \nu^{\alpha,\hat{l}}_s \rho_s ,  \partial_{z_\alpha} g_n \bigg\rangle_{L^2(\R^d)}  W_s^{\hat{l}}\\
    &  - \frac{1}{2}  \int\limits_0^t \bigg\langle  ( \sigma_s^2 +  \nu_s^2 ) \rho_s  ,  \Delta  g _n \bigg\rangle_{L^2(\R^d)} \Id s . 
  \end{align*}
  The first term, can be simply estimated by
  \begin{align*}
  \int\limits_0^t   \bigg \langle (k*u_s) \rho_s, \nabla g_n \bigg \rangle_{L^2(\R^{d})}  \Id s  
  &\le C \norm{k}_{L^\infty(\R^d)}^2 + \int\limits_{0}^{t} \qv{\rho_s, |\nabla g_n|^2}_{L^2(\R^d)} \Id s \\
  &\le C \norm{k}_{L^\infty(\R^d)}^2 + C \int\limits_{0}^{t} \qv{\rho_s, g_n}_{L^2(\R^d)} \Id s . 
  \end{align*}
  
  For the other term we obtain 
  \begin{equation*}
  \int\limits_0^t \bigg\langle ( \sigma_s^2 +  \nu_s^2 ) \rho_s   , 
\Delta g _n \bigg\rangle_{L^2(\R^d)} \Id s  
 \le  C \int\limits_0^t \int_{\R^d} \rho_s(x)   \Id x \Id s  
\le CT  ,
\end{equation*}   
where we used the uniform bound of \(|\Delta g_n|\).  
Now, we take the square and then the expectation to arrive at
\begin{align}
\begin{split} \label{eq: moment_est_aux1}
\E \big(  \qv{\rho_t, g_n}_{L^2(\R^d)} ^2 \big) 
\le&\; \qv{\rho_0, g_n}_{L^2(\R^d)} 
+ C T^2 \norm{k}_{L^\infty(\R^d)}^4 
+ C T^2 +C t^{\frac{1}{2}} \int\limits_{0}^{t} \E \big( \qv{\rho_s, g_n}_{L^2(\R^d)}^2 \big) \Id s  \\
&+\E \bigg(  \bigg| \sum\limits_{\alpha=1}^d  \sum\limits_{\hat{l}=1}^{\tilde{m}}  \int\limits_0^t \bigg\langle    \nu^{\alpha,\hat{l}}_s \rho_s ,  \partial_{z_\alpha} g_n \bigg\rangle_{L^2(\R^d)}  W_s^{\hat{l}}\bigg|^2 \bigg) . 
\end{split}
\end{align}
Using the BDG-inequality to estimate the stochastic integral we arrive at
\begin{align*}
&\; \E \bigg(  \bigg| \sum\limits_{\alpha=1}^d  \sum\limits_{\hat{l}=1}^{\tilde{m}}  \int\limits_0^t \bigg\langle    \nu^{\alpha,\hat{l}}_s \rho_s ,  \partial_{z_\alpha} g_n \bigg\rangle_{L^2(\R^d)}  W_s^{\hat{l}}\bigg|^2 \bigg)  \\
\le &\; \sum\limits_{\alpha,\beta=1}^d  \sum\limits_{\hat{l}=1}^{\tilde{m}} \E \bigg(  \int\limits_0^t   \bigg\langle    \nu^{\alpha,\hat{l}}_s \rho_s ,  \partial_{z_\alpha} g_n \bigg\rangle_{L^2(\R^d)}  \bigg\langle    \nu^{\beta,\hat{l}}_s \rho_s ,  \partial_{z_\beta} g_n \bigg\rangle_{L^2(\R^d)}  \Id  s \bigg) \\
\le &\; C(\norm{\nu}_{L^\infty(\R^d)}, \tilde{m} ,d ) \sum\limits_{\alpha=1}^d  \E \bigg(  \int\limits_0^t   \bigg| \langle   \rho_s ,  \partial_{z_\alpha} g_n \rangle_{L^2(\R^d)} \bigg|^2  \Id  s \bigg) \\
\le &\; C(\norm{\nu}_{L^\infty(\R^d)}, \tilde{m} ,d ) \; 
\E \bigg(  \int\limits_0^t  \langle   \rho_s , | \nabla g_n|^2 \rangle_{L^2(\R^d)}  \Id  s \bigg) \\
\le &\; C(\norm{\nu}_{L^\infty(\R^d)}, \tilde{m}, d )   \int\limits_0^t  \E \bigg(    \langle   \rho_s , g_n\rangle_{L^2(\R^d)}  \bigg)   \Id  s . 
\end{align*}
An application of Gronwall's lemma proves 
\begin{equation*}
    \sup\limits_{0 \le t \le T} \E\bigg( \bigg|\int_{\R^d} \rho_t(z) |z|^2 \Id z \bigg|^2 \,  \bigg) \le C(T,\sigma,\nu,\norm{k}_{L^\infty(\R^d)}) \bigg( \int_{\R^d} \rho_0(z) |z|^2 \Id z \bigg)^2.
\end{equation*}
Now, we can use this inequality to estimate the case, where the supremum is inside the expectation. Till inequality~\ref{eq: moment_est_aux1} we follow the same steps but with the difference that the supremum is now inside the expectation. We notice that we can apply Doob's maximal inequality to estimate the stochastic integral with the previous bound. Thus, obtaining the bound 
\begin{equation*}
\E\bigg ( \sup\limits_{0\le t \le T}  \bigg(\int_{\R^d} \rho_t(z) |z|^2 \Id z \bigg)^2 \bigg) 
 \le  C(T,\sigma,\nu,\norm{k}_{L^\infty(\R^d)}) \bigg(\int_{\R^d} \rho_0(x) |z|^2 \Id z\bigg)^2. 
 \end{equation*}

\end{proof}

We are ready to prove that the non-local \(d\)-dimensional SPDE~\eqref{eq: d_dim_spde} has a unique solution. 

\begin{proposition}[Existence of \(d\)-dimensional SPDE \(\rho\)] \label{prop: existence_d_spde}
Let \(k\in L^2(\R^d) \cap L^\infty(\R^d)\). Then the SPDE~\eqref{eq: d_dim_spde} has a unique solution solution in the space \( L^2_{\cF^W}([0,T];H^1(\R^{d}))\) satisfying 
\begin{equation} \label{eq: spde_h1_estimate}
    \norm{\rho}_{ L^2_{\cF^W}([0,T];H^1(\R^{d}))} 
\le C(T,\delta,d,\norm{k}_{L^2(\R^d) }, \norm{k}_{L^\infty(\R^d)}) \norm{\rho_0}_{L^2(\R^d)} . 
\end{equation}

\end{proposition}
\begin{proof}
Let us proof the existence by a Picard-Lindelöf iteration. Let us define for \(n \in \N\) the following SPDE 
\begin{align}\label{eq: existence_spde_aux1}
\begin{split} 
\Id \rho^n_t
=& \;  \nabla \cdot (( k*\rho^{n-1}_t) \rho^n_t ))  \Id t 
- \nabla \cdot (\nu_t  \rho^n \Id W_t)  \\
&\; + \frac{1}{2}  \sum\limits_{\alpha,\beta=1}^d \partial_{z_{\alpha}}\partial_{z_{\beta}} \bigg( ( [\sigma_t\sigma_t^{\mathrm{T}}]_{(\alpha,\beta)} + [\nu_t \nu_t^{\mathrm{T}}]_{(\alpha,\beta)} ) \rho^n_t \bigg)  \Id t
\end{split}
\end{align}
with initial value \(\rho_0\) and \(\rho^0 = \rho_0\). The initial value holds for all SPDE's in this proof. Then the SPDE is linear and by the estimate 
\begin{equation} \label{eq: exisetence_spde_aux2}
\norm{\nabla \cdot ((k*\rho^0) u )}_{H^{-1}(\R^d)} 
\le \norm{k*\rho^0}_{L^\infty(\R^d)} \norm{u}_{L^2(R^d)}
\le \norm{k}_{L^\infty(\R^d)} \norm{u}_{L^2(R^d)}
\end{equation}
for \(u \in L^2_{\cF^W}([0,T];H^1(\R^{d}) )\) we have a solution of the SPDE by~\cite[Theorem 5]{Krylov1999AnAA} in the case \(n=1\), therein. 
Furthermore, it is easy consequence of the divergence structure and the maximum principle that \(\rho^1\) is non-negative and satisfies mass conservation 
\begin{equation*}
  \norm{\rho^1_t}_{L^1(\R^d)} = \norm{\rho_0}_{L^1(\R^d)} = 1, \quad \P\text{-a.s.},
 \end{equation*}
  for \(t \in [0,T]\). A more detailed proof is given in~\cite{nikolaev2023hk} in a similar setting. Hence, since \(\rho^1\) has measure one uniform in time and \(\Omega\), we can derive, using the same arguments for each \(n\in \N\), a non-negative solution that satisfies mass conservation and the uniform estimate 
\begin{equation} \label{eq: sode_h1_picard_estimate}
\norm{\rho^n}_{ L^2_{\cF^W}([0,T];H^1(\R^{d}))} 
\le C(T,\delta,d) \norm{\rho_0}_{L^2(\R^d)}
\end{equation}
as provided by~\cite[Theorem 5.1]{Krylov1999AnAA}. Additionally, \(\rho^n\) satisfies the \(L^2\)-bound~\eqref{eq: uniform_L2_bound} uniform in \(n \in \N\).
Let us consider the difference of the SPDE's 
\begin{align*}
&\; \Id (\rho^n_t- \rho_t^{n-1}) \\
=& \; \nabla \cdot (( k*\rho^{n-1}_t)( \rho^n_t- \rho_t^{n-1}) + ( k*(\rho_t^{n-1}-\rho^{n-2}_t)) \rho^{n-1}_t )   \Id t  - \nabla \cdot (\nu_t  (\rho^n_t- \rho^{n-1}_t) \Id W_t)  \\
& + \frac{1}{2}  \sum\limits_{\alpha,\beta=1}^d \partial_{z_{\alpha}}\partial_{z_{\beta}} \bigg( ( [\sigma_t\sigma_t^{\mathrm{T}}]_{(\alpha,\beta)} + [\nu_t \nu_t^{\mathrm{T}}]_{(\alpha,\beta)} ) (\rho^n_t-\rho^{n-1}_t) \bigg)  \Id t. 
\end{align*}
  Applying It{\^o}'s formula~\cite{krylov2010ito}, we obtain
  \begin{align*}
&\; \norm{\rho_t^n-\rho^{n-1}_t}_{L^2(\R^d)}^2 \\
= &\; -2 \int\limits_0^t \int_{\R^d} (( k*\rho^{n-1}_s) (\rho^n_s - \rho^{n-1}_s) - (k*(\rho^{n-1}_s - \rho^{n-2}_s))\rho^{n-1}_s  )  \cdot \nabla  (\rho^n_s - \rho^{n-1}_s)\Id z \Id s  \\
&-   \sum\limits_{\alpha,\beta=1}^d \int\limits_0^t \int_{\R^d} \partial_{z_{\beta}}  \bigg( ( [\sigma_s\sigma_s^{\mathrm{T}}]_{(\alpha,\beta)} + [\nu_s \nu_s^{\mathrm{T}}]_{(\alpha,\beta)} ) (\rho^n_s- \rho^{n-1}_s) \bigg) \partial_{z_{\alpha}} (\rho^n_s - \rho^{n-1}_s) \Id z \Id s  \\
&+ \sum\limits_{\hat{l}=1}^{\tilde{m}}\int\limits_0^t \int_{\R^d} \bigg|\sum\limits_{\beta=1}^d \partial_{z_{\beta}} (\nu_s^{\beta,\hat{l}} (\rho^n_s- \rho^{n-1}_s) ) \bigg|^2 \Id z \Id s  \\
&+  2 \sum\limits_{\hat{l}=1}^{\tilde{m}} \sum\limits_{\beta=1}^d  \int\limits_0^t  \int_{\R^d} (\rho^n_s-\rho_s^{n-1}) \partial_{z_{\beta}} (\nu_s^{\beta,\hat{l}} (\rho^n_s- \rho^{n-1}_s) ) \Id x \Id  W_s^{\hat{l}} . 
  \end{align*}
Similar as before the stochastic integral vanishes by the divergence free assumption on \(\nu\).  
Again using the divergence-free assumption~\ref{ass: diffusion_coef}~\eqref{item: divergence_free_common_noise} we find  
\begin{align} \label{eq: d_spde_prood_aux1}
\begin{split}
&\;  \sum\limits_{\hat{l}=1}^{\tilde{m}}\int\limits_0^t \int_{\R^d} \bigg|\sum\limits_{\beta=1}^d \partial_{z_{\beta}} (\nu_s^{\beta,\hat{l}} (\rho^n_s- \rho^{n-1}_s) ) \bigg|^2 \Id  z \Id s   \\
 =&\;  \sum\limits_{\hat{l}=1}^{\tilde{m}} \sum\limits_{\alpha,\beta=1}^d \int\limits_0^t \int_{\R^d} \partial_{z_{\alpha}} (\nu_s^{\alpha,\hat{l}} (\rho^n_s- \rho^{n-1}_s) ) \partial_{z_{\beta}} (\nu_s^{\beta,\hat{l}} (\rho^n_s- \rho^{n-1}_s) )  \Id  z \Id s   \\
 =&\;  \sum\limits_{\hat{l}=1}^{\tilde{m}} \sum\limits_{\alpha,\beta=1}^d \int\limits_0^t \int_{\R^d} \nu_s^{\alpha,\hat{l}} \nu_s^{\beta,\hat{l}}   \partial_{z_{\alpha}} (\rho^n_s- \rho^{n-1}_s)   \partial_{z_{\beta}}(\rho^n_s- \rho^{n-1}_s)   \Id z \Id s  . 
 \end{split}
\end{align}

However, we also have 
 \begin{align*}
&\; \sum\limits_{\alpha,\beta=1}^d \int\limits_0^t \int_{\R^d}  \partial_{z_{\beta}}  \bigg( ( [\sigma_s\sigma_s^{\mathrm{T}}]_{(\alpha,\beta)} + [\nu_s \nu_s^{\mathrm{T}}]_{(\alpha,\beta)} ) (\rho^n_s - \rho^{n-1}_s) \bigg) \partial_{z_{\alpha}} (\rho^n_s - \rho^{n-1}_s)  \Id z  \Id s  \\
= &\;   \sum\limits_{\alpha,\beta=1}^d  \int\limits_0^t \int_{\R^d} \partial_{z_{\beta}}  \bigg(  [\sigma_s\sigma_s^{\mathrm{T}}]_{(\alpha,\beta)} + [\nu_s \nu_s^{\mathrm{T}}]_{(\alpha,\beta)} \bigg) (\rho^n_s  - \rho^{n-1}_s)  \partial_{z_{\alpha}} (\rho^n_s -  \rho^{n-1}_s)  \Id z \Id s \\
&+  \sum\limits_{\alpha,\beta=1}^d  \int\limits_0^t \int_{\R^d} ( [\sigma_s\sigma_s^{\mathrm{T}}]_{(\alpha,\beta)} + [\nu_s \nu_s^{\mathrm{T}}]_{(\alpha,\beta)}) \partial_{z_{\beta}}  (\rho^n_s - \rho^{n-1}_s)  \partial_{z_{\alpha}} (\rho^n_s - \rho^{n-1}_s)  \Id z \Id s .  
 \end{align*}
The last term with \(\nu\) is exactly the term \eqref{eq: d_spde_prood_aux1} and therefore both term cancel. 
For the first term we use the uniform bound on the derivatives and Young's inequality to estimated it by 
\begin{align*}
&\; \sum\limits_{\alpha,\beta=1}^d \bigg| \int\limits_0^t \int_{\R^d} \partial_{z_{\beta}}  \bigg(  [\sigma_s\sigma_s^{\mathrm{T}}]_{(\alpha,\beta)} + [\nu_s \nu_s^{\mathrm{T}}]_{(\alpha,\beta)} \bigg) (\rho^n_s - \rho^{n-1}_s)  \partial_{z_{\alpha}} (\rho^n_s - \rho^{n-1}_s)  \Id z \Id s\bigg| \\
&\;\le  C(d,\delta) \norm{\rho^n_s - \rho^{n-1}_s}_{L^2(\R^d)}^2 + \frac{\delta}{2} \norm{\nabla(\rho^n_s - \rho^{n-1}_s)}_{L^2(\R^d)}^2.
\end{align*}
Combining the last estimates we arrive at 
\begin{align*}
&\; \norm{\rho^n_t - \rho^{n-1}_t}_{L^2(\R^d)}^2 \\
= &\; -2 \int\limits_0^t \int_{\R^d} (( k*\rho^{n-1}_s) (\rho^n_s - \rho^{n-1}_s) - (k*(\rho^{n-1}_s-\rho^{n-2}_s))\rho^{n-1}_s )  \cdot \nabla  (\rho^n_s - \rho^{n-1}_s) \Id z \Id s  \\
& - \sum\limits_{\alpha,\beta=1}^d  \int\limits_0^t \int_{\R^d}  [\sigma_s\sigma_s^{\mathrm{T}}]_{(\alpha,\beta)}  \partial_{z_{\beta}}  (\rho^n_s - \rho^{n-1}_s)  \partial_{z_{\alpha}} (\rho^n_s - \rho^{n-1}_s)  \Id z \Id s    \\
&  + \frac{\delta}{2} \int\limits_0^t  \norm{\nabla(\rho^n_s - \rho^{n-1}_s)}_{L^2(\R^d)}^2  \Id s +  C(d,\delta )\int\limits_0^t\norm{\rho^n_s - \rho^{n-1}_s}_{L^2(\R^d)}^2 \Id s \\
\le &\; 4 \delta   \int\limits_0^t \norm{ k*\rho^{n-1}_s}_{L^\infty(\R^d)}^2 \norm{ \rho^{n}_s-\rho^{n-1}_s}_{L^2(\R^d)}^2  \Id s 
+  C(d,\delta) \int\limits_0^t \norm{\rho_s^n- \rho_s^{n-1}}_{L^2(\R^d)}^2  \Id s \\
&+ 4\delta \int\limits_0^t \norm{ k(\rho^{n-1}_s-\rho^{n-2}_s)}_{L^\infty(\R^d)}^2  \norm{\rho^{n-1}_s}_{L^2(\R^d)}^2 \Id s \\
\le &\;   C(d,\delta,\norm{k}_{L^\infty})  \int\limits_0^t \norm{ \rho^{n}_s-\rho^{n-1}_s}_{L^2(\R^d)}^2  \Id s + 4 \delta \norm{k}_{L^2(\R^d)}^2 \int\limits_0^t \norm{ \rho^{n-1}_s-\rho^{n-2}_s}_{L^2(\R^d)}^2  \Id s, 
\end{align*}
 where we used the uniform \(L^2\) bound~\eqref{eq: uniform_L2_bound} and ellipticity condition in the last step. 
 Taking expectation and using Gronwall's lemma, 
 we find 
 \begin{equation*}
 \E \big( \norm{\rho^n_t - \rho^{n-1}_t}_{L^2(\R^d)}^2  \big) 
 \le  4 \delta \norm{k}_{L^2(\R^d)}^2 e^{C(d,\delta,\norm{k}_{L^\infty} )T }  \int\limits_0^t \E( \norm{ \rho^{n-1}_s-\rho^{n-2}_s}_{L^2(\R^d)}^2 )   \Id s .
 \end{equation*}
 The standard Picard--Lindelöf iteration implies that 
 \begin{equation*}
\sum\limits_{n=1}^\infty \sup\limits_{0 \le t \le T}   \E \big( \norm{\rho^n_t - \rho^{n-1}_t}_{L^2(\R^d)}^2 \big)  < \infty
 \end{equation*}
 and therefore we can find a function \(\rho \in  L^2_{\cF^W}([0,T];H^1(\R^{d}))\) such that 
 \begin{equation} \label{eq: picard_iteration_conv}
 \lim\limits_{n\to \infty} \norm{\rho^n-\rho}_{ L^2_{\cF^W}([0,T];H^1(\R^{d}))} = 0.
 \end{equation}
Furthermore, \(\rho\) satisfies mass conservation and therefore by similar arguments as before there exists a solution \(\hat{\rho}\) of the linear SPDE 
\begin{align*}
\Id \hat{\rho}_t
=& \;  \nabla \cdot (( k*\rho_t) \hat{\rho}_t ))  \Id t 
- \nabla \cdot (\nu_t  \hat{\rho}_t \Id W_t)  \\
&\; + \frac{1}{2}  \sum\limits_{\alpha,\beta=1}^d \partial_{z_{\alpha}}\partial_{z_{\beta}} \bigg( ( [\sigma_t\sigma_t^{\mathrm{T}}]_{(\alpha,\beta)} + [\nu_t \nu_t^{\mathrm{T}}]_{(\alpha,\beta)} ) \hat{\rho}_t \bigg)  \Id t,
\end{align*}
which also satisfies the \(L^2\)-bound~\eqref{eq: uniform_L2_bound} . 
Applying the inequality in~\cite[Theorem 5.1]{Krylov1999AnAA} and the \(L^2\)-bound~\eqref{eq: uniform_L2_bound} we find 
\begin{align*}
\norm{\hat{\rho}-\rho^n}_{L^2_{\cF^W}([0,T];H^1(\R^{d}))} 
&\le \norm{(k*(\rho-\rho^n))\hat{\rho}}_{L^2_{\cF^W}([0,T];L^2(\R^{d}))}  \\
&\le C  \norm{\rho-\rho^n}_{L^2_{\cF^W}([0,T];L^2(\R^{d}))} \\
&{\xrightarrow[]{n \to \infty}} 0 . 
\end{align*}
Hence \(\rho = \hat{\rho}\) and we have a solution to the non-linear SPDE~\eqref{eq: d_dim_spde}. The properties of \(\rho\) are a direct consequence of the properties of \(\rho^n\) and the strong convergence in \(L^2_{\cF^W}([0,T];H^1(\R^{d})) \).
\end{proof}
\begin{remark}
We cannot use the inequality provided by~\cite[Theorem 5.1]{Krylov1999AnAA} directly since our SPDE is nonlinear. Therefore, we need to rely on a Gronwall argument and estimate the \(L^2(\R^d)\)-norm directly. We purposely choose this presentation of the existence of the SPDE~\eqref{eq: d_dim_spde} because we will need a stability result in Section~\ref{sec: stability_spde}, which follows the same idea. Alternatively, we could have verified the monotonicity conditions of Wei and Röckner~\cite{LiuWei2015SPDE}.
\end{remark}

\begin{lemma} \label{lemma: aux_chaotic_spde}
Let \(\rho_t\) be a solution of the \(d\)-dimensional SPDE~\eqref{eq: d_dim_spde} provided by Proposition~\ref{prop: existence_d_spde}. 
Then the tensor product \(\rho^{\otimes N}(x) = \prod\limits_{i=1}^N \rho(x_i)\) for \(x=(x_1,\ldots,x_N) \in \R^{dN}\) solves the \(dN\)-dimensional SPDE~\eqref{eq: chaotic_spde} in the sense of Definition~\ref{def: solution_of_chaotic_spde}. 
\end{lemma}
\begin{proof}
By the structure of the solution we immediately obtain the regularity necessary in Definition~\ref{def: solution_of_chaotic_spde}. 
Let \(\varphi_{i} \in C_c^\infty(\R^d)\) for \(i=1,\ldots,N)\) be smooth functions. Then for \(\varphi(x)= \prod\limits_{i=1}^N \varphi(x_i)\) we obtain
\begin{equation*}
\langle  \rho_t^{\otimes N},  \varphi \rangle_{L^2(\R^{dN})}
=   \prod\limits_{i=1}^N  \langle\rho_t ,\varphi_i \rangle_{L^2(\R^{d})}. 
\end{equation*}
The right hand side is now a product of It\^{o} processes. Hence applying the It\^{o} formula to the function \(f(y) = y_1 y_2 \cdots y_N\) for \(y \in \R^N\) we obtain 
\begin{align*}
&\; \Id \bigg(\prod\limits_{i=1}^N  \langle\rho_t ,\varphi_i \rangle_{L^2(\R^{d})} \bigg)  \\
=&\;  \sum\limits_{i=1}^N  \prod\limits_{j\neq i}^N  \langle\rho_t ,\varphi_j \rangle_{L^2(\R^{d})}  \Id \langle\rho_t ,\varphi_i \rangle_{L^2(\R^{d})} 
+ \frac{1}{2}\sum\limits_{\substack{i,j=1 \\ i \neq j}}^N \prod\limits_{ \substack{q \neq i \\ q \neq j}}^N  \langle\rho_t ,\varphi_q \rangle_{L^2(\R^{d})}  \Id \langle  \langle\rho_t ,\varphi_i \rangle_{L^2(\R^{d})},   \langle\rho_t ,\varphi_j \rangle_{L^2(\R^{d})}  \rangle . 
\end{align*}
Computing the quadratic variation we find 
\begin{align*}
&\; \langle  \langle\rho_t ,\varphi_i \rangle_{L^2(\R^{d})},   \langle\rho_t ,\varphi_j \rangle_{L^2(\R^{d})}  \rangle \\
=&\; \bigg\langle \sum\limits_{\alpha=1}^d \sum\limits_{l=1}^{\tilde{m}} \int\limits_0^t \bigg\langle    \nu^{\alpha,l}(s,x_i) \rho_s ,  \partial_{x_{i,\alpha}} \varphi_i \bigg\rangle_{L^2(\R^{d})}  W_s^{l} , \sum\limits_{\beta=1}^d \sum\limits_{\hat{l}=1}^{\tilde{m}} \int\limits_0^t \bigg\langle    \nu^{\beta,\hat{l}}(s,x_j) \rho_s ,  \partial_{x_{j,\beta}} \varphi_j \bigg\rangle _{L^2(\R^{d})} W_s^{\hat{l}}
\bigg\rangle \\
=&\; \sum\limits_{\alpha,\beta=1}^d \sum\limits_{\hat{l}=1}^{\tilde{m}} \int\limits_0^t \bigg\langle    \nu^{\alpha,\hat{l}}(s,x_i) \rho_s ,  \partial_{x_{i,\alpha}} \varphi_i \bigg\rangle_{L^2(\R^{d})} \bigg\langle    \nu^{\beta,\hat{l}}(s,x_j) \rho_s ,  \partial_{x_{j,\beta}} \varphi_j \bigg\rangle_{L^2(\R^{d})} \Id s , 
\end{align*}
which implies the following form for the covariation term 
\begin{equation*}
\frac{1}{2}\sum\limits_{\substack{i,j=1 \\ i \neq j}}^N  \sum\limits_{\alpha,\beta=1}^d \sum\limits_{\hat{l}=1}^{\tilde{m}}  \int\limits_0^t    \langle  \nu^{\alpha,\hat{l}}(s,x_i)  \nu^{\beta,\hat{l}}(s,x_j) \bigg( \prod\limits_{q=1}^N\rho_s\bigg)(\cdot),  \partial_{x_{i,\alpha}} \partial_{x_{j,\beta}} \varphi \rangle_{L^2(\R^d)} \Id s . 
\end{equation*}
Plugging in the dynamic \(\langle \rho_t, \varphi \rangle\), we find 
\begin{align*}
&\;  \sum\limits_{i=1}^N  \prod\limits_{j\neq i}^N  \langle\rho_t ,\varphi_j \rangle_{L^2(\R^{d})}  \Id \langle\rho_t ,\varphi_i \rangle_{L^2(\R^{d})} 
\\
= &\sum\limits_{i=1}^N \int\limits_0^t   \prod\limits_{j\neq i}^N  \langle\rho_s ,\varphi_j \rangle_{L^2(\R^{d})} \bigg(  \bigg \langle (k*\rho_s)(x_i) \rho_s, \nabla \varphi_i \bigg \rangle_{L^2(\R^{d})}  \\
   &    + \frac{1}{2} \sum\limits_{\alpha,\beta=1}^d  \bigg\langle   [\sigma(s,\cdot)\sigma(s,\cdot)^{\mathrm{T}}]_{(\alpha,\beta)} \rho_s, \partial_{x_{i,\beta}}  \partial_{x_{i,\alpha}} \varphi_i \bigg\rangle_{L^2(\R^{d})}  \\
    &+  \frac{1}{2} \sum\limits_{\alpha,\beta=1}^d  \bigg\langle   [\nu(s,x_i)\nu(s,x_i)^{\mathrm{T}}]_{(\alpha,\beta)} \rho_s, \partial_{x_{i,\beta}}  \partial_{x_{i,\alpha}} \varphi_i \bigg\rangle_{L^2(\R^{d})} \bigg) \Id s \\
&+ \sum\limits_{i=1}^N  \sum\limits_{\alpha=1}^d \sum\limits_{\hat{l}=1}^{\tilde{m}} \int\limits_{0}^t \prod\limits_{j\neq i}^N   \langle\rho_s ,\varphi_j \rangle_{L^2(\R^{d})}   \bigg\langle    \nu^{\alpha,\hat{l}}(s,x_i) \rho_s^{\otimes{N}} ,  \partial_{x_{i,\alpha}} \varphi_i \bigg\rangle_{L^2(\R^d)} W_s^{\hat{l}}. 
\end{align*}
Putting the product into one integral over \(\R^{dN}\) as before we see that \(\prod\limits_{i=1}^N \rho_t(x_i)\) satisfies equation~\eqref{eq: def_chaotic_spde} for \(\varphi \in C_c^\infty(\R^{dN})\) of the form \(\varphi(x)= \prod\limits_{i=1}^N \varphi(x_i)\). Becasuse functions of this form are dense in \(L^2(\R^{dN})\) the claim follows by the regularity of \(\rho\). 
\end{proof}

\begin{remark}
    At the end of the section let us make some comments on the regularity of the interaction kernel \(k\). First, we are concerned with the global existence of solutions, meaning for arbitrary \(T >0 \). Otherwise, the conditional propagation of chaos results seems not powerful, since we assume our initial data is independent and, therefore, the independence should at least propagate for some small \(T>0\).  Hence, to demonstrate global existence, the map \(u \mapsto (k*u) \rho \) needs to be a continuous linear map. 
    One way to enforce this is for \(k*u\) to be uniformly bounded in the probability space and in the time variable, as discussed in~\cite{HuangQui2021,nikolaev2023hk}.
    Additionally, It\^{o}'s formula for the \(L^p\)-norms requires the condition \(p \ge 2\). Consequently, lacking concrete properties of \(k\), we require \(k\) to belong to \(L^{p'}(\R^d)\), where \(p'\) is the conjugated exponent.
    
    However, in order to apply the exponential law of large numbers~\cite[Theorem 3]{JabinWang2018} in Section~\ref{sec: relative_entropy_smooth_setting}, we require a bounded force \(k\). Therefore, the condition \(k \in L^2(\R^d)\cap L^\infty(\R^d)\) appears to be necessary and minimal. 
\end{remark}

\section{Relative entropy in the smooth case}
\label{sec: relative_entropy_smooth_setting}
In this section we assume that all coefficients \(\sigma, \nu, k\) are smooth uniform in time. More precisely, for all indices \(\alpha, \beta, \tilde{l} \) we assume 
\begin{equation} \label{eq: smooth_case_uniform-estimate}
\sup\limits_{0\le t \le T} \norm{[\sigma(t,\cdot) \sigma^{\mathrm{T}}(t,\cdot)]_{(\alpha,\beta)}}_{C^\infty(\R^d)} +
\sup\limits_{0\le t \le T}\norm{\nu^{\beta, \tilde{l}} (t, \cdot)}_{C^\infty(\R^d)} +
\norm{k}_{C^\infty(\R^d)} \le  C 
\end{equation}
for some constant \(C>0\) and Assumption~\ref{ass: diffusion_coef} holds.  
The idea is to demonstrate the relative entropy estimate in cases where \(\rho^N,\rho^{\otimes N}\) solve the SPDE's~\ref{eq: liouville_SPDE},~\ref{eq: chaotic_spde} , not in the weak sense against a test function, but rater pointwise. The key to achieving this lies in the Sobolev embedding, which mandates at least two derivatives to properly interpret the pointwise second derivatives in the SPDE's~\ref{eq: liouville_SPDE},~\ref{eq: chaotic_spde}. 

\begin{proposition}  \label{prop: solution_ito_proces_rep}
For each \(N\in \N\)
the solutions  \(\rho_t^N\) and \(\rho_t^{\otimes N}\) are smooth and there exists versions, which have an Ito process representation. 

Hence, there exists a set \(\tilde{\Omega}\) with \(\P(\tilde{\Omega})  = 1 \) such that for all \((t,x) \in [0,T] \times \R^d\) we have 
\begin{align*}
&\; \rho^N(\omega, t, x) \\
=&\; \rho_0(x) 
+ \sum\limits_{i=1}^N \int\limits_0^t   \nabla_{x_i} \cdot \bigg( \frac{1}{N} \sum\limits_{j=1}^N k(x_i-x_j) \rho^N(\omega,s,x) \bigg) \Id s  \\
&- \sum\limits_{i=1}^N  \sum\limits_{\alpha=1}^d \sum\limits_{\hat{l}=1}^{\tilde{m}} \int\limits_0^t   \nu^{\alpha,\hat{l}}(s,x_i) \partial_{x_{i,\alpha}} \rho^{ N}(\omega,s,x)   W_s^{\hat{l}}   \\
&+ \frac{1}{2} \sum\limits_{i=1}^N \sum\limits_{\alpha,\beta=1}^d  \int\limits_0^t \partial_{x_{i,\alpha}}\partial_{x_{i,\beta}} \bigg( [\sigma(s,x_i)\sigma(s,x_i)^{\mathrm{T}}]_{(\alpha,\beta)}   \rho^N(\omega,s,x) \bigg)  \Id s\\ 
&+   \frac{1}{2} \sum\limits_{i,j=1}^N \sum\limits_{\alpha,\beta=1}^d  \int\limits_0^t \partial_{x_{i,\alpha}}\partial_{x_{j,\beta}} \bigg(   + [\nu(s,x_i)\nu(s,x_j)^{\mathrm{T}}]_{(\alpha,\beta)} \rho^N(\omega,s,x) \bigg)  \Id s
\end{align*}
and 
\begin{align*}
&\; \rho^{\otimes N}(\omega, t, x) \\
=&\; \rho_0(x) 
+ \sum\limits_{i=1}^N \int\limits_0^t   \nabla_{x_i} \cdot \bigg( (k*\rho)(\omega,t,x) \rho^{\otimes N}(\omega,s,x) \bigg) \Id s  \\
&- \sum\limits_{i=1}^N  \sum\limits_{\alpha=1}^d \sum\limits_{\hat{l}=1}^{\tilde{m}} \int\limits_0^t   \nu^{\alpha,\hat{l}}(s,x_i) \partial_{x_{i,\alpha}} \rho^{\otimes N}(\omega,s,x)   W_s^{\hat{l}}   \\
&+  \frac{1}{2} \sum\limits_{i,j=1}^N \sum\limits_{\alpha,\beta=1}^d  \int\limits_0^t \partial_{x_{i,\alpha}}\partial_{x_{i,\beta}} \bigg(  [\sigma(s,x_i)\sigma(s,x_i)^{\mathrm{T}}]_{(\alpha,\beta)} \rho^{\otimes N}(\omega,s,x) \bigg)  \Id s\\
&+   \frac{1}{2} \sum\limits_{i,j=1}^N \sum\limits_{\alpha,\beta=1}^d  \int\limits_0^t \partial_{x_{i,\alpha}}\partial_{x_{j,\beta}} \bigg(  [\nu(s,x_i)\nu(s,x_j)^{\mathrm{T}}]_{(\alpha,\beta)} \rho^{\otimes N}(\omega,s,x) \bigg)  \Id s . 
\end{align*}
\end{proposition}
\begin{proof}
Let us fix \(N \in \N\) and choose a \(n \in \N\) such that \(n> 2+\frac{dN}{2}\). 
We start with the linear \(N\)-particle SPDE. 
Similar to the non-smooth case, we can obtain a solution
\begin{equation*}
\rho^N \in  L^2_{\cF^W}([0,T];H^1(\R^{d})), 
\end{equation*}
which satisfies mass conservation. 
Since all coefficients are smooth we can write the SPDE's in non-divergence form. Therefore, we can apply~\cite[Theorem 5.1 and Remark 5.6]{Krylov1999AnAA} in the case \(p=2\) and \(n\) as above to obtain  
\begin{equation*}
\rho^N \in L^2_{\cF^W}([0,T];H^n(\R^{d})). 
\end{equation*}
The It\^{o} representation, then immideatly follows from the regularity and~\cite[Theorem 4.3(e)]{RozovskyBorisL2018SES}. 
For the non-linear SPDE \(\rho^{\otimes N}\) we repeat the arguments with one slight modification. After obtaining the solution by the same steps as in Section~\ref{sec: existence_ass_spde},
we linearize the equation by fixing \(\rho\) in the convolution \(k*\rho\). Now, the coefficients are again smooth and for arbitrary multiindex \(\gamma\) we have 
\begin{equation*}
\norm{\partial^{\gamma}(k*\rho_t)}_{L^\infty}
\le \norm{k}_{C^{|\gamma|}(\R^d)}
\end{equation*}
by Young's inequality and mass conservation of \(\rho\). This provides a smooth solution \(\rho\) and consequently, by the same argument as in Lemma~\ref{lemma: aux_chaotic_spde}, we obtain a smooth solution \(\rho^{\otimes N}\), which again has the It\^{o} representation by ~\cite[Theorem 4.3(e)]{RozovskyBorisL2018SES}. 

\end{proof}
For each \(x \in \R^{dN}\) we obtain two process \(\rho_t^N(x)\) and \(\rho_t^{\otimes N}(x)\). Consequently, we can utilize It\^{o}'s formula to derive an evolution of the relative entropy \(\mathcal{H}(\rho_t^N \vert \rho^{\otimes N }_t )\). 
To begin, we split the integrand \(x \log(x/y)\) of the relative entropy into two terms \(x\log(x)\) and \(-x\log(y)\). Subsequently, we apply It\^{o}'s formula to each function separately and combine them later.

\begin{lemma}[It\^{o}'s formula for \(x \log(x) \)]
\label{lem: ito_aux_1} 
Let \(\rho_t^N\) be the smooth solution provided by Proposition~\ref{prop: solution_ito_proces_rep}. Then, 
\begin{align} \label{eq: ito_aux_1}
&\; \Id \rho_t^N(x)  \log(\rho_t^{N})(x)  \nonumber\\
=&\; (\log(\rho_t^{N}(x))  +1) \bigg( \sum\limits_{i=1}^N   \nabla_{x_i} \cdot \bigg( \frac{1}{N} \sum\limits_{j=1}^N k(x_i-x_j) \rho_t^N(x) \bigg) \Id t  \nonumber \\
& - \sum\limits_{i=1}^N  \sum\limits_{\alpha=1}^d \sum\limits_{\hat{l}=1}^{\tilde{m}}  \nu_t^{\alpha,\hat{l}}(x_{i}) \partial_{x_{i,\alpha}} \rho_t^{N}(x)   \Id W_t^{\hat{l}}    +   \frac{1}{2} \sum\limits_{i=1}^N \sum\limits_{\alpha,\beta=1}^d  \partial_{x_{i,\alpha}}\partial_{x_{i,\beta}}  ( [\sigma_t(x_i)\sigma_t(x_i)^{\mathrm{T}}]_{(\alpha,\beta)}\rho_t^N ) \Id t  \nonumber  \\
&+   \frac{1}{2} \sum\limits_{i,j=1}^N \sum\limits_{\alpha,\beta=1}^d  \partial_{x_{i,\alpha}}\partial_{x_{j,\beta}} ( [\nu_t(x_i)\nu_t(x_j)^{\mathrm{T}}]_{(\alpha,\beta)} \rho_t^N) \Id t  \bigg) \nonumber \\
&+ \frac{\rho_t^N}{2}  \sum\limits_{\hat{l}=1}^{\tilde{m}}  \bigg( \sum\limits_{i=1}^N  \sum\limits_{\alpha=1}^d  \nu_t^{\alpha,\hat{l}}(x_i) \partial_{x_{i,\alpha}} \log(\rho_t^{ N}(x))  \bigg)^2  \Id  t. 
\end{align}
In particular, we can choose a set \(\tilde{\Omega} \subseteq \Omega\) with \(\P(\tilde{\Omega})\) such that equality~\eqref{eq: ito_aux_1} holds for all \((t,x)\) and the integrands on the right hand side of equality~\eqref{eq: ito_aux_1} are product measurable in \((\omega,s,x)\).
\end{lemma}
\begin{proof}
Let us fix \(x\in \R^{dN}\). Then for almost all \(\omega\) and all \(t \ge 0\) we have 
\begin{equation*}
\Id \rho_t^N(x)  \log(\rho_t^{N})(x)   
= (\log(\rho_t^{N})(x)   +1)\Id \rho_t^N + \frac{1}{2 \rho_t^N(x) }  \Id \langle \rho_t^N(x)  \rangle .
\end{equation*}
For the quadratic variation we obtain 
\begin{align*}
\Id \langle \rho_t^{N}(x) \rangle 
&=  \Id \bigg \langle \sum\limits_{i=1}^N  \sum\limits_{\alpha=1}^d \sum\limits_{\hat{l}=1}^{\tilde{m}}   \nu^{\alpha,\hat{l}}(t,x_i) \partial_{x_{i,\alpha}} \rho_t^{ N} (x)  W_t^{\hat{l}} \bigg\rangle \\
&=   \sum\limits_{i,j=1}^N  \sum\limits_{\alpha,\beta=1}^d \sum\limits_{\hat{l}=1}^{\tilde{m}}  (\nu^{\alpha,\hat{l}}(t,x_i) \partial_{x_{i,\alpha}} \rho_t^{ N}(x)  ) (\nu^{\alpha,\hat{l}}(t,x_j) \partial_{x_{j,\beta}} \rho_t^{ N}(x)  )  \Id  t. 
\end{align*}
Consequently, combining the quadratic variation with the dynamic given in Proposition~\ref{prop: solution_ito_proces_rep} we obtain
\begin{align*}
&\; \Id \rho_t^N(x)  \log(\rho_t^{N})(x)  \\
=&\; (\log(\rho_t^{N}(x))  +1) \bigg( \sum\limits_{i=1}^N   \nabla_{x_i} \cdot \bigg( \frac{1}{N} \sum\limits_{j=1}^N k(x_i-x_j) \rho_t^N(x) \bigg) \Id t  \\
&- \sum\limits_{i=1}^N  \sum\limits_{\alpha=1}^d \sum\limits_{\hat{l}=1}^{\tilde{m}}  \nu_t^{\alpha,\hat{l}}(x_{i}) \partial_{x_{i,\alpha}} \rho_t^{N}(x)   \Id W_t^{\hat{l}}   \\
&+   \frac{1}{2} \sum\limits_{i=1}^N \sum\limits_{\alpha,\beta=1}^d  \partial_{x_{i,\alpha}}\partial_{x_{i,\beta}} ( [\sigma_t(x_i)\sigma_t(x_i)^{\mathrm{T}}]_{(\alpha,\beta)}  \rho_t^N(x) )  \Id t \bigg) \\
 &+ \frac{1}{2} \sum\limits_{i,j=1}^N \sum\limits_{\alpha,\beta=1}^d  \partial_{x_{i,\alpha}}\partial_{x_{j,\beta}} ([\nu_t(x_i)\nu_t(x_j)^{\mathrm{T}}]_{(\alpha,\beta)}  \rho_t^N(x) )  \Id t \bigg) \\
&+ \frac{1}{2\rho_t^N(x)} \sum\limits_{i,j=1}^N  \sum\limits_{\alpha,\beta=1}^d \sum\limits_{\hat{l}=1}^{\tilde{m}}  (\nu_t^{\alpha,\hat{l}}(x_i) \partial_{x_{i,\alpha}} \rho_t^{ N}(x) ) (\nu_t^{\alpha,\hat{l}}(x_j) \partial_{x_{j,\beta}} \rho_t^{ N}(x) )  \Id  t. 
\end{align*}
Now, we can rewrite the last term into 
\begin{align*}
&\; \frac{1}{2\rho_t^N(x)}  \sum\limits_{\hat{l}=1}^{\tilde{m}}  \bigg( \sum\limits_{i=1}^N  \sum\limits_{\alpha=1}^d  (\nu^{\alpha,\hat{l}}(t,x_i) \partial_{x_{i,\alpha}} \rho_t^{ N}(x) ) \bigg)^2  \Id  t \\
=&\;  \frac{\rho_t^N(x)}{2}  \sum\limits_{\hat{l}=1}^{\tilde{m}}  \bigg( \sum\limits_{i=1}^N  \sum\limits_{\alpha=1}^d  \nu^{\alpha,\hat{l}}(t,x_i) \partial_{x_{i,\alpha}} \log(\rho_t^{ N}(x))  \bigg)^2  \Id  t.
\end{align*}
Inserting it into the previous equation we obtain equation \ref{eq: ito_aux_1}. For the measurability statement we just notice that the integrands are continuous in time and space \(R^d\). Therefore we can apply Kolmogorov's continuity criteria in \((t,x)\) to obtain a continuous version of the stochastic integral, which implies the continuity of the right hand side of equation~\ref{eq: ito_aux_1} and consequently the product measurability. 
\end{proof}

Similar, we obtain an expression for \(\rho_t^N (x) \log(\rho_t^{\otimes N}(x))\). 

\begin{lemma}[It\^{o}'s formula for \(x \log(y)\)]
\label{lem: ito_aux_2}
Let \(\rho_t^{\otimes N}\) be the smooth solution provided by Proposition~\ref{prop: solution_ito_proces_rep}. Then, 

\begin{align} \label{eq: ito_aux_2}
&\; \Id \rho_t^N (x) \log(\rho_t^{\otimes N}(x)) \nonumber \\
= &\; \log(\rho_t^{\otimes N}(x))  \bigg( \sum\limits_{i=1}^N   \nabla_{x_i} \cdot \bigg( \frac{1}{N} \sum\limits_{j=1}^N k(x_i-x_j) \rho_t^N(x) \bigg) \Id t \nonumber \\
&  - \sum\limits_{i=1}^N  \sum\limits_{\alpha=1}^d \sum\limits_{\hat{l}=1}^{\tilde{m}}  \nu_t^{\alpha,\hat{l}}(x_i) \partial_{x_{i,\alpha}} \rho_t^{N}(x)   \Id W_s^{\hat{l}}   +   \frac{1}{2} \sum\limits_{i=1}^N \sum\limits_{\alpha,\beta=1}^d  \partial_{x_{i,\alpha}}\partial_{x_{i,\beta}} ( [\sigma_t(x_i)\sigma_t(x_i)^{\mathrm{T}}]_{(\alpha,\beta)}  \rho_t^N(x)  \Id t   \nonumber \\
&+   \frac{1}{2} \sum\limits_{i,j=1}^N \sum\limits_{\alpha,\beta=1}^d  \partial_{x_{i,\alpha}}\partial_{x_{j,\beta}} ( [\nu_t(x_i)\nu_t(x_j)^{\mathrm{T}}]_{(\alpha,\beta)} ) \rho_t^N(x) )   \Id t \bigg)  \nonumber \\
&+ \frac{\rho_t^N(x)}{\rho_t^{\otimes N}(x)} 
\bigg( \sum\limits_{i=1}^N   \nabla_{x_i} \cdot \bigg( (k*\rho_t)(x_i) \rho_t^{\otimes N}(x) \bigg) \Id t   - \sum\limits_{i=1}^N  \sum\limits_{\alpha=1}^d \sum\limits_{\hat{l}=1}^{\tilde{m}}  \nu_t^{\alpha,\hat{l}}(x_i) \partial_{x_{i,\alpha}} \rho_t^{\otimes N}(x)   \Id W_s^{\hat{l}}   \nonumber \\
&+   \frac{1}{2} \sum\limits_{i=1}^N \sum\limits_{\alpha,\beta=1}^d  \partial_{x_{i,\alpha}}\partial_{x_{i,\beta}} ( [\sigma_t(x_i)\sigma_t(x_i)^{\mathrm{T}}]_{(\alpha,\beta)} \rho_t^{\otimes N}(x) )  \Id t \nonumber \\
&+  \frac{1}{2} \sum\limits_{i,j=1}^N \sum\limits_{\alpha,\beta=1}^d  \partial_{x_{i,\alpha}}\partial_{x_{j,\beta}} ( [\nu_t(x_i)\nu_t(x_j)^{\mathrm{T}}]_{(\alpha,\beta)} ) \rho_t^{\otimes N}(x) )  \Id t \bigg)  \nonumber \\
&+ \rho_t^{ N}(x)
 \sum\limits_{i,j=1}^N  \sum\limits_{\alpha,\beta=1}^d \sum\limits_{\hat{l}=1}^{\tilde{m}}  (\nu_t^{\alpha,\hat{l}}(x_i) \partial_{x_{i,\alpha}} \log(\rho_t^{ N})  (x)) (\nu_t^{\alpha,\hat{l}}(x_j) \partial_{x_{j,\beta}} \log( \rho_t^{\otimes N})  (x))  \Id  t \nonumber \\
&-  \frac{\rho_t^{N}(x)}{2}
 \sum\limits_{\hat{l}=1}^{\tilde{m}}  \bigg(\sum\limits_{i=1}^N  \sum\limits_{\alpha=1}^d  \nu_t^{\alpha,\hat{l}}(x_i) \partial_{x_{i,\alpha}} \log(\rho_t^{\otimes N}(x) )  \bigg)^2 \Id  t. 
\end{align}
In particular, we can choose a set \(\tilde{\Omega} \subseteq \Omega\) with \(\P(\tilde{\Omega})\) such that equality~\eqref{eq: ito_aux_1} holds for all \((t,x)\) and the integrands on the right hand side of equality~\eqref{eq: ito_aux_1} are product measurable in \((\omega,s,x)\).
\end{lemma}
\begin{proof}
Applying It\^{o}'s formula to \(x \log(y)\) we obtain 
\begin{align*}
\Id \rho_t^N (x) \log(\rho_t^{\otimes N}(x))  
=&\;  \log(\rho_t^{\otimes N}(x)) \Id \rho_t^{N}(x) + \frac{\rho_t^{N}(x)}{\rho_t^{\otimes N}(x)} \Id \rho_t^{\otimes N}(x)
+ \frac{1}{\rho_t^{\otimes N}(x)} \Id \langle \rho_t^{N}(x),\rho_t^{\otimes N}(x) \rangle  \\
&- \frac{\rho_t^{N}(x)}{2 (\rho_t^{\otimes N}(x))^2} \Id \langle \rho_t^{\otimes N}(x) \rangle  
\end{align*}
Computing the quadratic variation we find 
\begin{align*}
\Id \langle \rho_t^{\otimes N}(x) \rangle 
&=  \Id \bigg \langle \sum\limits_{i=1}^N  \sum\limits_{\alpha=1}^d \sum\limits_{\hat{l}=1}^{\tilde{m}}   \nu^{\alpha,\hat{l}}(t,x_i) \partial_{x_{i,\alpha}} \rho_t^{\otimes N}(x)   W_t^{\hat{l}} \bigg\rangle \\
&=   \sum\limits_{i,j=1}^N  \sum\limits_{\alpha,\beta=1}^d \sum\limits_{\hat{l}=1}^{\tilde{m}}  (\nu_t^{\alpha,\hat{l}}(x_i) \partial_{x_{i,\alpha}} \rho_t^{\otimes N}(x) ) (\nu_t^{\alpha,\hat{l}}(x_j) \partial_{x_{j,\beta}} \rho_t^{\otimes N}(x) )  \Id  t \\ 
&= \sum\limits_{\hat{l}=1}^{\tilde{m}}  \bigg(\sum\limits_{i=1}^N  \sum\limits_{\alpha=1}^d  \nu_t^{\alpha,\hat{l}}(x_i) \partial_{x_{i,\alpha}} \rho_t^{\otimes N}(x)  \bigg)^2 \Id  t,
\end{align*}
where we used the fact that \(\Id \langle W_t^{l}, W_t^{\tilde{l}} \rangle = \delta_{l, \tilde{l}}  \Id t \). 
Now, by similar arguments we obtain
\begin{align*}
&\; \Id \langle \rho_t^{N}(x),\rho_t^{\otimes N}(x) \rangle \\
=&\;  \Id \bigg \langle \sum\limits_{i=1}^N  \sum\limits_{\alpha=1}^d \sum\limits_{l=1}^{\tilde{m}}   \nu_t^{\alpha,l}(x_i) \partial_{x_{i,\alpha}} \rho_t^{N}(x)   W_t^{l},  \sum\limits_{j=1}^N  \sum\limits_{\beta=1}^d \sum\limits_{\hat{l}=1}^{\tilde{m}}   \nu_t^{\beta,\hat{l}}(x_j) \partial_{x_{j,\beta}} \rho_t^{\otimes N}(x)   W_t^{\hat{l}}  \bigg\rangle \\
=&\;  \sum\limits_{i,j=1}^N  \sum\limits_{\alpha,\beta=1}^d \sum\limits_{\hat{l}=1}^{\tilde{m}}  (\nu_t^{\alpha,\hat{l}}(x_i) \partial_{x_{i,\alpha}} \rho_t^{ N} (x)) (\nu_t^{\alpha,\hat{l}}(x_j) \partial_{x_{j,\beta}} \rho_t^{\otimes N} (x))  \Id  t. 
\end{align*}
Consequently, we arrive at 
\begin{align*}
&\; \Id \rho_t^N (x) \log(\rho_t^{\otimes N}(x))  \\
= &\; \log(\rho_t^{\otimes N}(x))  \bigg( \sum\limits_{i=1}^N   \nabla_{x_i} \cdot \bigg( \frac{1}{N} \sum\limits_{j=1}^N k(x_i-x_j) \rho_t^N(x) \bigg) \Id t  - \sum\limits_{i=1}^N  \sum\limits_{\alpha=1}^d \sum\limits_{\hat{l}=1}^{\tilde{m}}  \nu_t^{\alpha,\hat{l}}(x_i) \partial_{x_{i,\alpha}} \rho_t^{N}(x)   \Id W_t^{\hat{l}}   \\
&+   \frac{1}{2} \sum\limits_{i=1}^N \sum\limits_{\alpha,\beta=1}^d  \partial_{x_{i,\alpha}}\partial_{x_{i,\beta}} ( [\sigma_t(x_i)\sigma_t(x_i)^{\mathrm{T}}]_{(\alpha,\beta)}  \rho_t^N(x)  \Id t   \\
&+   \frac{1}{2} \sum\limits_{i,j=1}^N \sum\limits_{\alpha,\beta=1}^d  \partial_{x_{i,\alpha}}\partial_{x_{j,\beta}} ( [\nu_t(x_i)\nu_t(x_j)^{\mathrm{T}}]_{(\alpha,\beta)} ) \rho_t^N(x) )   \Id t \bigg)  \\
&+ \frac{\rho_t^N(x)}{\rho_t^{\otimes N}(x)} 
\bigg( \sum\limits_{i=1}^N   \nabla_{x_i} \cdot \bigg( (k*\rho_t)(x_i) \rho_t^{\otimes N}(x) \bigg) \Id t  - \sum\limits_{i=1}^N  \sum\limits_{\alpha=1}^d \sum\limits_{\hat{l}=1}^{\tilde{m}}  \nu_t^{\alpha,\hat{l}}(x_i) \partial_{x_{i,\alpha}} \rho_t^{\otimes N}(x)   \Id W_t^{\hat{l}}   \\
&+   \frac{1}{2} \sum\limits_{i=1}^N \sum\limits_{\alpha,\beta=1}^d  \partial_{x_{i,\alpha}}\partial_{x_{i,\beta}} ( [\sigma_t(x_i)\sigma_t(x_i)^{\mathrm{T}}]_{(\alpha,\beta)} \rho_t^{\otimes N}(x) )  \Id t  \\
&+  \frac{1}{2} \sum\limits_{i,j=1}^N \sum\limits_{\alpha,\beta=1}^d  \partial_{x_{i,\alpha}}\partial_{x_{j,\beta}} ( [\nu_t(x_i)\nu_t(x_j)^{\mathrm{T}}]_{(\alpha,\beta)} ) \rho_t^{\otimes N}(x) )  \Id t \bigg) \\
&+\frac{1}{\rho_t^{\otimes N}(x)} 
 \sum\limits_{i,j=1}^N  \sum\limits_{\alpha,\beta=1}^d \sum\limits_{\hat{l}=1}^{\tilde{m}}  (\nu_t^{\alpha,\hat{l}}(x_i) \partial_{x_{i,\alpha}} \rho_t^{ N} (x)) (\nu_t^{\alpha,\hat{l}}(x_j) \partial_{x_{j,\beta}} \rho_t^{\otimes N} (x))  \Id  t \\
&-  \frac{\rho_t^{N}(x)}{2 (\rho_t^{\otimes N}(x))^2}
 \sum\limits_{\hat{l}=1}^{\tilde{m}}  \bigg(\sum\limits_{i=1}^N  \sum\limits_{\alpha=1}^d  \nu_t^{\alpha,\hat{l}}(x_i) \partial_{x_{i,\alpha}} \rho_t^{\otimes N}(x)  \bigg)^2 \Id  t. 
\end{align*}
For the last two term we find 
\begin{align*}
&\; \frac{1}{\rho_t^{\otimes N}(x)} 
 \sum\limits_{i,j=1}^N  \sum\limits_{\alpha,\beta=1}^d \sum\limits_{\hat{l}=1}^{\tilde{m}}  (\nu_t^{\alpha,\hat{l}}(x_i) \partial_{x_{i,\alpha}} \rho_t^{ N} (x)) (\nu_t^{\alpha,\hat{l}}(x_j) \partial_{x_{j,\beta}} \rho_t^{\otimes N} (x))  \Id  t \\
&-  \frac{\rho_t^{N}(x)}{2 (\rho_t^{\otimes N}(x))^2}
 \sum\limits_{\hat{l}=1}^{\tilde{m}}  \bigg(\sum\limits_{i=1}^N  \sum\limits_{\alpha=1}^d  \nu_t^{\alpha,\hat{l}}(x_i) \partial_{x_{i,\alpha}} \rho_t^{\otimes N}(x)  \bigg)^2 \Id  t \\
=&\; \rho_t^{ N}(x)
 \sum\limits_{i,j=1}^N  \sum\limits_{\alpha,\beta=1}^d \sum\limits_{\hat{l}=1}^{\tilde{m}}  (\nu_t^{\alpha,\hat{l}}(x_i) \partial_{x_{i,\alpha}} \log(\rho_t^{ N})  (x)) (\nu_t^{\alpha,\hat{l}}(x_j) \partial_{x_{j,\beta}} \log( \rho_t^{\otimes N})  (x))  \Id  t \\
&-  \frac{\rho_t^{N}(x)}{2}
 \sum\limits_{\hat{l}=1}^{\tilde{m}}  \bigg(\sum\limits_{i=1}^N  \sum\limits_{\alpha=1}^d  \nu_t^{\alpha,\hat{l}}(x_i) \partial_{x_{i,\alpha}} \log(\rho_t^{\otimes N}(x) )  \bigg)^2 \Id  t , 
\end{align*}
which with the same measurability arguments as in Lemma~\ref{lem: ito_aux_1} implies our claim. 
\end{proof}

In the next step we want to combine both Lemma~\ref{lem: ito_aux_1} and Lemma~\ref{lem: ito_aux_2}. We make a crucial observation that the difference of the last term in equation~\eqref{eq: ito_aux_1} and the last two terms in equation~\eqref{eq: ito_aux_2} creating a square, i.e. 
\begin{align} \label{eq: convariation_perfect_square}
\begin{split}
&\;  \frac{\rho_t^N}{2}  \sum\limits_{\hat{l}=1}^{\tilde{m}}  \bigg( \sum\limits_{i=1}^N  \sum\limits_{\alpha=1}^d  \nu_t^{\alpha,\hat{l}}(x_i) \partial_{x_{i,\alpha}} \log(\rho_t^{ N}(x))  \bigg)^2  \Id  t \\
&- \rho_t^{ N}(x)
 \sum\limits_{i,j=1}^N  \sum\limits_{\alpha,\beta=1}^d \sum\limits_{\hat{l}=1}^{\tilde{m}}  (\nu_t^{\alpha,l}(x_i) \partial_{x_{i,\alpha}} \log(\rho_t^{ N})  (x)) (\nu_t^{\alpha,\hat{l}}(x_j) \partial_{x_{j,\beta}} \log( \rho_t^{\otimes N})  (x))  \Id  t \\
&+  \frac{\rho_t^{N}(x)}{2}
 \sum\limits_{\hat{l}=1}^{\tilde{m}}  \bigg(\sum\limits_{i=1}^N  \sum\limits_{\alpha=1}^d  \nu_t^{\alpha,\hat{l}}(x_i) \partial_{x_{i,\alpha}} \log(\rho_t^{\otimes N}(x) )  \bigg)^2 \Id  t \\
 =&\; \frac{\rho_t^{N}(x)}{2}
 \sum\limits_{\hat{l}=1}^{\tilde{m}}  \bigg(\sum\limits_{i=1}^N  \sum\limits_{\alpha=1}^d  \nu_t^{\alpha,\hat{l}}(x_i) \partial_{x_{i,\alpha}} \log\bigg(\frac{\rho_t^{N}(x)}{\rho_t^{\otimes N} (x)} \bigg )  \bigg)^2 \Id  t.
 \end{split}
\end{align}
As a result we obtain the following dynamic 
\begin{align*}
&\; \Id (\rho_t^N(x) (\log(\rho_t^N(x) -\log(\rho_t^{\otimes N}(x)) )   \\
= &\; \bigg(\log\bigg( \frac{\rho_t^N(x)}{\rho_t^{\otimes N}(x)}\bigg) + 1 \bigg)  \bigg( \sum\limits_{i=1}^N   \nabla_{x_i} \cdot \bigg( \frac{1}{N} \sum\limits_{j=1}^N k(x_i-x_j) \rho_t^N(x) \bigg) \Id t \\
&  - \sum\limits_{i=1}^N  \sum\limits_{\alpha=1}^d \sum\limits_{l=1}^{\tilde{m}}  \nu_t^{\alpha,l}(x_i) \partial_{x_{i,\alpha}} \rho_t^{N}(x)   \Id W_s^{l}   \\
&+   \frac{1}{2} \sum\limits_{i=1}^N \sum\limits_{\alpha,\beta=1}^d  \partial_{x_{i,\alpha}}\partial_{x_{i,\beta}} ( [\sigma_t(x_i)\sigma_t(x_i)^{\mathrm{T}}]_{(\alpha,\beta)}  \rho_t^N(x)  \Id t   \\
&+   \frac{1}{2} \sum\limits_{i,j=1}^N \sum\limits_{\alpha,\beta=1}^d  \partial_{x_{i,\alpha}}\partial_{x_{j,\beta}} ( [\nu_t(x_i)\nu_t(x_j)^{\mathrm{T}}]_{(\alpha,\beta)} ) \rho_t^N(x) )   \Id t \bigg)  \\
&+ \frac{\rho_t^N(x)}{\rho_t^{\otimes N}(x)} 
\bigg( \sum\limits_{i=1}^N   \nabla_{x_i} \cdot \bigg( (k*\rho_t)(x_i) \rho_t^{\otimes N}(x) \bigg) \Id t   - \sum\limits_{i=1}^N  \sum\limits_{\alpha=1}^d \sum\limits_{\hat{l}=1}^{\tilde{m}}  \nu_t^{\alpha,\hat{l}}(x_i) \partial_{x_{i,\alpha}} \rho_t^{\otimes N}(x)   \Id W_s^{\hat{l}}   \\
&+  \frac{1}{2} \sum\limits_{i=1}^N \sum\limits_{\alpha,\beta=1}^d  \partial_{x_{i,\alpha}}\partial_{x_{i,\beta}} ( [\sigma_t(x_i)\sigma_t(x_i)^{\mathrm{T}}]_{(\alpha,\beta)}  \rho_t^{\otimes N}(x)  \Id t   \\
&+   \frac{1}{2} \sum\limits_{i,j=1}^N \sum\limits_{\alpha,\beta=1}^d  \partial_{x_{i,\alpha}}\partial_{x_{j,\beta}} ( [\nu_t(x_i)\nu_t(x_j)^{\mathrm{T}}]_{(\alpha,\beta)} ) \rho_t^{\otimes N}(x) )   \Id t \bigg)  \\
&+ \frac{\rho_t^{N}(x)}{2}
 \sum\limits_{\hat{l}=1}^{\tilde{m}}  \bigg(\sum\limits_{i=1}^N  \sum\limits_{\alpha=1}^d  \nu_t^{\alpha,\hat{l}}(x_i) \partial_{x_{i,\alpha}} \log\bigg(\frac{\rho_t^{N}(x)}{\rho_t^{\otimes N} (x)} \bigg )  \bigg)^2 \Id  t. 
\end{align*}
After integrating over \(\R^{dN}\), the divergence free assumption of \(\nu\) and the stochastic Fubini Theorem~\cite{veraar2012}  kills the stochastic integrals. 
Hence, applying Fubini's theorem (everything is product measurable) for the Lebesgue integrals we find 
\begin{align*}
&\; \mathcal{H}(\rho_t^{N}  \vert \rho_t^{\otimes N})  \\
=&\;   \int_0^t \int_{\R^{dN}}  \bigg(\log\bigg( \frac{\rho_s^N(x)}{\rho_s^{\otimes N}(x)}\bigg) + 1 \bigg)  \bigg( \sum\limits_{i=1}^N   \nabla_{x_i} \cdot \bigg( \frac{1}{N} \sum\limits_{j=1}^N k(x_i-x_j) \rho_s^N(x) \bigg)  \\
&+    \frac{1}{2} \sum\limits_{i=1}^N \sum\limits_{\alpha,\beta=1}^d  \partial_{x_{i,\alpha}}\partial_{x_{i,\beta}} ( [\sigma_s(x_i)\sigma_s(x_i)^{\mathrm{T}}]_{(\alpha,\beta)}  \rho_s^N(x)   \\
&+   \frac{1}{2} \sum\limits_{i,j=1}^N \sum\limits_{\alpha,\beta=1}^d  \partial_{x_{i,\alpha}}\partial_{x_{j,\beta}} ( [\nu_s(x_i)\nu_s(x_j)^{\mathrm{T}}]_{(\alpha,\beta)} ) \rho_s^N(x) )   \bigg)  \\
&-\frac{\rho_s^N(x)}{\rho_s^{\otimes N}(x)} 
\bigg( \sum\limits_{i=1}^N   \nabla_{x_i} \cdot \bigg( (k*\rho_s)(x_i) \rho_s^{\otimes N}(x) \bigg)\\ 
&- \frac{1}{2} \sum\limits_{i=1}^N \sum\limits_{\alpha,\beta=1}^d  \partial_{x_{i,\alpha}}\partial_{x_{i,\beta}} ( [\sigma_s(x_i)\sigma_s(x_i)^{\mathrm{T}}]_{(\alpha,\beta)}  \rho_s^{\otimes N}(x)    \\
&-   \frac{1}{2} \sum\limits_{i,j=1}^N \sum\limits_{\alpha,\beta=1}^d  \partial_{x_{i,\alpha}}\partial_{x_{j,\beta}} ( [\nu_s(x_i)\nu_s(x_j)^{\mathrm{T}}]_{(\alpha,\beta)} ) \rho_s^{\otimes N}(x) )   \bigg)  \\
&+ \frac{\rho_s^{N}(x)}{2}
 \sum\limits_{\hat{l}=1}^{\tilde{m}}  \bigg(\sum\limits_{i=1}^N  \sum\limits_{\alpha=1}^d  \nu_s^{\alpha,\hat{l}}(x_i) \partial_{x_{i,\alpha}} \log\bigg(\frac{\rho_s^{N}(x)}{\rho_s^{\otimes N} (x)} \bigg )  \bigg)^2   \Id x  \Id  s  . 
\end{align*}
We notice that the constant one in the first term vanishes by the divergence structure of the equation and the integration over the whole domain \(\R^{dN}\). 
In the next step let us use the cancellation property of \(\sigma\) and \(\nu\) to rewrite the second order differential operator
\begin{align*}
& \; \frac{1}{2} \sum\limits_{i,j=1}^N \sum\limits_{\alpha,\beta=1}^d  
  \partial_{x_{i,\alpha}}\partial_{x_{j,\beta}} \bigg( \bigg( \sum\limits_{l=1}^{\tilde{m}} \sigma_s^{\alpha,l}(x_i)\sigma_s^{\beta,l}(x_j) \delta_{i,j} + \sum\limits_{\hat{l}=1}^{\tilde{m}} \nu_s^{\alpha,\hat{l}}(x_i)\nu_s^{\beta,\hat{l}}(x_j) \bigg) \rho_s^N(x) \bigg)  \\
  &=\frac{1}{2} \sum\limits_{i,j=1}^N \sum\limits_{\alpha,\beta=1}^d  
  \partial_{x_{i,\alpha}} \bigg( \bigg( \sum\limits_{l=1}^{\tilde{m}} \sigma_s^{\alpha,l}(x_i)\sigma_s^{\beta,l}(x_j) \delta_{i,j} + \sum\limits_{\hat{l}=1}^{\tilde{m}} \nu_s^{\alpha,\hat{l}}(x_i)\nu_s^{\beta,\hat{l}}(x_j) \bigg) \partial_{x_{j,\beta}} \rho_s^N(x) \bigg)  .
 \end{align*}
 The same inequality holds if \(\rho^N\) is replaced by \(\rho^{\otimes N}\). Hence, if we only look at the terms containing \(\nu\) we arrive at the following expression for the coefficient \(\nu\), 
 \begin{align*}
 &\;  \frac{1}{2} \sum\limits_{i,j=1}^N \sum\limits_{\alpha,\beta=1}^d \sum\limits_{\hat{l}=1}^{\tilde{m}} \int_{\R^{dN}}  \log\bigg( \frac{\rho_s^N(x)}{\rho_s^{\otimes N}(x)}\bigg) 
  \partial_{x_{i,\alpha}} \bigg( \nu_s^{\alpha,\hat{l}}(x_i)\nu_s^{\beta,\hat{l}}(x_j)   \partial_{x_{j,\beta}} \rho_s^N(x) \bigg) \Id x \\
 &  - \frac{1}{2} \sum\limits_{i,j=1}^N \sum\limits_{\alpha,\beta=1}^d \sum\limits_{\hat{l}=1}^{\tilde{m}} \int_{\R^{dN}} \frac{\rho_s^N(x)}{\rho_s^{\otimes N}(x)}   \partial_{x_{i,\alpha}}\bigg( \nu_s^{\alpha,\hat{l}}(x_i)\nu_s^{\beta,\hat{l}}(x_j)  \partial_{x_{j,\beta}}  \rho_s^{\otimes N}(x) \bigg)   \Id x \\
  =&\; - \frac{1}{2} \sum\limits_{i,j=1}^N \sum\limits_{\alpha,\beta=1}^d \sum\limits_{\hat{l}=1}^{\tilde{m}} \int_{\R^{dN}}   \partial_{x_{i,\alpha}}  \log\bigg( \frac{\rho_s^N(x)}{\rho_s^{\otimes N}(x)}\bigg) 
 \nu_s^{\alpha,\hat{l}}(x_i)\nu_s^{\beta,\hat{l}}(x_j)   \partial_{x_{j,\beta}} \rho_s^N(x) \Id x  \\
 &+ \frac{1}{2} \sum\limits_{i,j=1}^N \sum\limits_{\alpha,\beta=1}^d \sum\limits_{\hat{l}=1}^{\tilde{m}} \int_{\R^{dN}} \partial_{x_{i,\alpha}} \bigg( \frac{\rho_s^N(x)}{\rho_s^{\otimes N}(x)}  \bigg) \nu_s^{\alpha,\hat{l}}(x_i)\nu_s^{\beta,\hat{l}}(x_j)  \partial_{x_{j,\beta}}  \rho_s^{\otimes N}(x) \Id x\\ 
 =&\; - \frac{1}{2} \sum\limits_{i,j=1}^N \sum\limits_{\alpha,\beta=1}^d \sum\limits_{\hat{l}=1}^{\tilde{m}} \int_{\R^{dN}}  \rho_s^N(x)  \partial_{x_{i,\alpha}}  \log\bigg( \frac{\rho_s^N(x)}{\rho_s^{\otimes N}(x)}\bigg) 
 \nu_s^{\alpha,\hat{l}}(x_i)\nu_s^{\beta,\hat{l}}(x_j)   \partial_{x_{j,\beta}} \log(\rho_s^N(x)) \Id x \\
 & + \frac{1}{2} \sum\limits_{i,j=1}^N \sum\limits_{\alpha,\beta=1}^d \sum\limits_{\hat{l}=1}^{\tilde{m}} \int_{\R^{dN}} \rho_s^N(x) \partial_{x_{i,\alpha}} \log \bigg( \frac{\rho_s^N(x)}{\rho_s^{\otimes N}(x)}  \bigg) \nu_s^{\alpha,\hat{l}}(x_i)\nu_s^{\beta,\hat{l}}(x_j)  \partial_{x_{j,\beta}}  \log(\rho_s^{\otimes N})(x) \Id x \\
  =&\; - \frac{1}{2} \sum\limits_{\hat{l}=1}^{\tilde{m}} \int_{\R^{dN}}   \rho_s^N(x) \bigg(  \sum\limits_{i=1}^N  \sum\limits_{\alpha=1}^d \nu_s^{\alpha,\hat{l}}(x_i)  \partial_{x_{i,\alpha}}   \log\bigg( \frac{\rho_s^N(x)}{\rho_s^{\otimes N}(x)}\bigg)\bigg)^2 \Id x . 
\end{align*}
Hence, we obtain exactly the same term as in the covariation calculations~\eqref{eq: convariation_perfect_square} but with a negative sign. Consequently, they vanish and we do not see any contribution of the common noise \(\nu\) in the relative entropy. 
Let us summarize our result in the following proposition. 

\begin{proposition}
Let \(\rho_t^{N} \), \(\rho_t^{\otimes N}\) be given by Proposition~\ref{prop: solution_ito_proces_rep}. Then we have the following representation of the reltive entropy
\begin{align*}
\mathcal{H}(\rho_t^{N}  \vert \rho_t^{\otimes N}) 
=&\;   \int_0^t  \int_{\R^{dN}}  \sum\limits_{i=1}^N  \log\bigg( \frac{\rho_s^N(x)}{\rho_s^{\otimes N}(x)}\bigg)   \nabla_{x_i} \cdot \bigg( \frac{1}{N} \sum\limits_{j=1}^N k(x_i-x_j) \rho_s^N(x) \bigg)  \\
&- \sum\limits_{i=1}^N  \frac{\rho_s^N(x)}{\rho_s^{\otimes N}(x)} 
  \nabla_{x_i} \cdot \bigg( (k*\rho_s)(x_i) \rho_s^{\otimes N}(x) \bigg)  \\
&+   \frac{1}{2} \sum\limits_{i=1}^N \sum\limits_{\alpha,\beta=1}^d  \log\bigg( \frac{\rho_s^N(x)}{\rho_s^{\otimes N}(x)}\bigg)    \partial_{x_{i,\alpha}} \bigg(  [\sigma_s(x_i)\sigma_s(x_i)^{\mathrm{T}}]_{(\alpha,\beta)}  \partial_{x_{i,\beta}} \rho_s^N(x) \bigg)\\
&- \frac{1}{2} \sum\limits_{i=1}^N \sum\limits_{\alpha,\beta=1}^d  \frac{\rho_s^N(x)}{	\rho_s^{\otimes N}(x)}\partial_{x_{i,\alpha}} \bigg(  [\sigma_s(x_i)\sigma_s(x_i)^{\mathrm{T}}]_{(\alpha,\beta)}   \partial_{x_{i,\beta}} \rho_s^{\otimes N}(x) \bigg) \Id x  \Id  s .
\end{align*}

\end{proposition}
\begin{remark}
 This phenomenon seems maybe strange but it somehow shows that the common noise has no effect on the expected relative entropy \(\E ( \mathcal{H}(\rho_t^{N}  \vert \rho_t^{\otimes N}) )\) as long as both measures are conditioned on the common noise. This is a crucial observation. If both measures are viewed under the information of the common noise \(W\), we expect that the particles behave similar as in the classical mean-field limit setting and this phenomenon is exactly reflected in the relative entropy. 
\end{remark}

From the representation we obtain the classical relative entropy bound. 

\begin{corollary} \label{cor: entropy_estimate}
Let \(\rho_t^{N} \), \(\rho_t^{\otimes N}\) be given by Proposition~\ref{prop: solution_ito_proces_rep}.
We have the following relative entropy inequality 
\begin{align*}
 \mathcal{H}(\rho_t^{N}  \vert \rho_t^{\otimes N})  
\le &\;  - \frac{\delta}{4}  \sum\limits_{i=1}^N \int\limits_0^t  \int_{\R^{dN}}  \rho_s^N(x)  \bigg| \nabla_{x_i} \log\bigg( \frac{\rho_s^N(x)}{\rho_s^{\otimes N}(x)}\bigg)\bigg|^2  \Id x \Id s  \\
&+ \frac{1}{\delta} \sum\limits_{i=1}^N  \int\limits_0^t  \rho_s^N(x)\bigg| \frac{1}{N} \sum\limits_{j=1}^N k(x_{i}-x_j)  - (k*\rho_s)(x_{i})\bigg|^2  \Id x  \Id s. 
\end{align*}
\end{corollary}

\begin{proof}
The proof is a standard entropy computation similar to the ones performed in~\cite{Lacker2023uniform,nikolaev2024}. Applying Young's inequality and the exchangeability of the particles we find 
\begin{align*}
&\; \int_{\R^{dN}}  \sum\limits_{i=1}^N  \log\bigg( \frac{\rho_s^N(x)}{\rho_s^{\otimes N}(x)}\bigg)   \nabla_{x_i} \cdot \bigg( \frac{1}{N} \sum\limits_{j=1}^N k(x_i-x_j) \rho_s^N(x) \bigg)  \\
&- \sum\limits_{i=1}^N  \frac{\rho_s^N(x)}{\rho_s^{\otimes N}(x)} 
  \nabla_{x_i} \cdot \bigg( (k*\rho_s)(x_i) \rho_s^{\otimes N}(x) \bigg) \Id  x \\
 = &\; \sum\limits_{i=1}^N \int_{\R^{dN}} \rho_s^N(x)  \nabla_{x_i} \log\bigg( \frac{\rho_s^N(x)}{\rho_s^{\otimes N}(x)}\bigg)  \cdot  \bigg( \frac{1}{N} \sum\limits_{j=1}^N k(x_i-x_j)  - (k*\rho_s)(x_i)\bigg)  \Id x  \\
\le &\;   \sum\limits_{i=1}^N \int_{\R^{dN}} \frac{\delta }{2} \rho_s^N(x)  \bigg| \nabla_{x_i} \log\bigg( \frac{\rho_s^N(x)}{\rho_s^{\otimes N}(x)}\bigg)\bigg|^2  + \frac{ \rho_s^N(x)}{2\delta}  \bigg| \frac{1}{N} \sum\limits_{j=1}^N k(x_i-x_j)  - (k*\rho_s)(x_i)\bigg|^2  \Id x  \\
\le & \; \sum\limits_{i=1}^N \int_{\R^{dN}} 
\frac{\delta }{4} \rho_s^N(x)  \bigg| \nabla_{x_i} \log\bigg( \frac{\rho_s^N(x)}{\rho_s^{\otimes N}(x)}\bigg)\bigg|^2  +  \frac{ \rho_s^N(x)}{\delta}  \bigg| \frac{1}{N} \sum\limits_{j=1}^N k(x_{i}-x_j)  - (k*\rho_s)(x_{i})\bigg|^2  \Id x  . 
\end{align*}
Using integration by parts and the ellipticity condition on \(\sigma\) we find 
\begin{align*}
&\; \frac{1}{2} \sum\limits_{i=1}^N \sum\limits_{\alpha,\beta=1}^d  \int_{\R^{dN}} \log\bigg( \frac{\rho_s^N(x)}{\rho_s^{\otimes N}(x)}\bigg)    \partial_{x_{i,\alpha}}\bigg(  [\sigma_s(x_i)\sigma_s(x_i)^{\mathrm{T}}]_{(\alpha,\beta)}   \partial_{x_{i,\beta}}  \rho_s^N(x) \bigg)\\
&- \frac{1}{2} \sum\limits_{i=1}^N \sum\limits_{\alpha,\beta=1}^d  \frac{\rho_s^N(x)}{	\rho_s^{\otimes N}(x)}\partial_{x_{i,\alpha}}  \bigg(  [\sigma_s(x_i)\sigma_s(x_i)^{\mathrm{T}}]_{(\alpha,\beta)} \partial_{x_{i,\beta}}   \rho_s^{\otimes N}(x) \bigg) \Id x \\
=&\; - \frac{1}{2} \sum\limits_{i=1}^N \sum\limits_{\alpha,\beta=1}^d  \int_{\R^{dN}} \rho_s^N(x)  \partial_{x_{i,\alpha}} \log\bigg( \frac{\rho_s^N(x)}{\rho_s^{\otimes N}(x)}\bigg)   [\sigma_s(x_i)\sigma_s(x_i)^{\mathrm{T}}]_{(\alpha,\beta)}   \partial_{x_{i,\beta}}  \log(\rho_s^N(x)) \\
&+ \frac{1}{2} \sum\limits_{i=1}^N \sum\limits_{\alpha,\beta=1}^d  \rho_s^N(x)  \partial_{x_{i,\alpha}} \log\bigg( \frac{\rho_s^N(x)}{\rho_s^{\otimes N}(x)}\bigg)   [\sigma_s(x_i)\sigma_s(x_i)^{\mathrm{T}}]_{(\alpha,\beta)} \partial_{x_{i,\beta}}   \log(\rho_s^{\otimes N}(x))  \Id x \\
=&\; - \frac{1}{2} \sum\limits_{i=1}^N \int_{\R^{dN}} \rho_s^N(x) \sum\limits_{\alpha,\beta=1}^d   \partial_{x_{i,\alpha}} \log\bigg( \frac{\rho_s^N(x)}{\rho_s^{\otimes N}(x)}\bigg)   [\sigma_s(x_i)\sigma_s(x_i)^{\mathrm{T}}]_{(\alpha,\beta)}   \partial_{x_{i,\beta}}  \log\bigg(\frac{\rho_s^N(x)}{\rho_s^{\otimes N}(x)} \bigg) \Id x\\
\le &\; -  \frac{\delta}{2}  \sum\limits_{i=1}^N \int_{\R^{dN}}  \rho_s^N(x)  \bigg| \nabla_{x_i} \log\bigg( \frac{\rho_s^N(x)}{\rho_s^{\otimes N}(x)}\bigg)\bigg|^2 \Id x . 
\end{align*} 
Combining both inequalities proves the corollary. 
\end{proof}

\begin{theorem} \label{theorem: main_estimate_smooth_setting}
In the smooth setting we have the following relative entropy bound between the conditional law of the particles system \(\rho^N\) and the the solution \(\rho^{\otimes N}\) of the SPDE~\eqref{eq: chaotic_spde}, 
\begin{equation}
 \E \bigg(\sup\limits_{0 \le t \le T } \mathcal{H}(\rho_t^{N}  \vert \rho_t^{\otimes N}) \bigg) \le C(T,\delta,\norm{k}_{L^\infty}) ,
\end{equation}
where \(C(T,\norm{k}_{L^\infty})\) depends on the finial time \(T\), the ellipticity constant \(\delta\) and \(\norm{k}_{L^\infty} \). 
\end{theorem}
\begin{remark}
    The results demonstrate that in the smooth case, the relative entropy estimates presented by Jabin and Wang~\cite{JabinWang2018} hold even in the presence of common noise. However, if our coefficients and interaction kernel are not smooth, we need to resort to an approximation argument. As discussed in Section~\ref{sec: stability_spde}, this becomes achievable when \(k\in L^2(\R^d) \cap L^\infty(\R^d) \). The challenging aspect lies in handling the non-linear SPDE~\ref{eq: d_dim_spde}. Consequently, we have reformulated the inquiry about a quantitative relative entropy estimate with common noise into a question concerning the stability of non-linear SPDEs. It is foreseeable that other results derived from the work of Jabin and Wang~\cite{JabinWang2018} can be adapted to the common noise setting by similar techniques.  
\end{remark}
\begin{proof}[Proof of Theorem\ref{theorem: main_estimate_smooth_setting}]
The proof is basically an application of~\cite[Theorem 3]{JabinWang2018}. Let us define the following process
\begin{equation*}
\psi(s,z,y) = \frac{1}{16 e \norm{k}_{L^\infty}} (k(z-y)-k*\rho_s(z))
\end{equation*}
and notice that \(\norm{\psi}_{L^\infty(\R^d)}
\le \frac{1}{2e}\) uniformly in time and the probability space \(\Omega\).
Then applying Corollary~\ref{cor: entropy_estimate} and~\cite[Lemma 1]{JabinWang2018} we find 
\begin{align*}
&\;  \E \bigg(\sup\limits_{0 \le t \le T } \mathcal{H}(\rho_t^{N}  \vert \rho_t^{\otimes N}) \bigg) \\
\le &\; \frac{1}{\delta} \sum\limits_{i=1}^N  \int\limits_0^T \E\bigg(  \rho_s^N(x)\bigg| \frac{1}{N} \sum\limits_{j=1}^N k(x_{i}-x_j)  - (k*\rho_s)(x_{i})\bigg|^2  \Id x \bigg) \Id s \\
\le&\;   \frac{16e \norm{k}_{L^\infty(\R^d)} }{\delta} \int\limits_0^T \E\bigg(  \mathcal{H}(\rho_s^{N}  \vert \rho_s^{\otimes N})  \bigg)  \\
&+  \frac{1}{N} \sum\limits_{i=1}^N \E\bigg( \log\bigg( \int_{\R^{dN} } \rho^{\otimes N}_s(x) \exp\bigg( \frac{1}{N} \sum\limits_{j_{1},j_{2}=1}^N  \psi(s,x_{i},x_{j_1}) \psi(s,x_{i},x_{j_2}) \bigg) \Id x \bigg) \bigg) \Id s. 
\end{align*}
Additionally, we have the following cancellation property for all \(s \in [0,T]\), 
\begin{equation*}
\int_{\R^d} \psi(s,z,y) \rho_s(y) = 0 , \quad \P\text{-a.e.}.   
\end{equation*} 
At this moment we can repeat the proof of~\cite[Theorem 3]{JabinWang2018}, since it is based on the cancellation property and combinatorial arguments. Hence, it can be performed path-wise and we arrive at 
\begin{equation} \label{eq: enviromental_relative_entropy_bound_aux1}
 \E\bigg( \sup\limits_{0 \le t \le T} \mathcal{H}(\rho_t^{N}  \vert \rho_t^{\otimes N}) \bigg)
 \le   \frac{16e \norm{k}_{L^\infty(\R^d)} }{\delta} 
 \int\limits_0^t \E\bigg(  \mathcal{H}(\rho_s^{N}  \vert \rho_s^{\otimes N})  \bigg)  
 + C  \Id s . 
\end{equation}
 Finally, Gronwall's lemma implies
\begin{equation*}
\sup\limits_{0 \le t \le T} \E(\mathcal{H}(\rho_t^{N}  \vert \rho_t^{\otimes N})) 
 \le   \frac{16e \norm{k}_{L^\infty(\R^d)  }T e^{CT} }{\delta} 
\end{equation*}
for some positive constant \(C>0\). Utilizing this estimate in~\eqref{eq: enviromental_relative_entropy_bound_aux1} we arrive at 
\begin{equation}
\E\bigg( \sup\limits_{0 \le t \le T} \mathcal{H}(\rho_t^{N}  \vert \rho_t^{\otimes N}) \bigg)
\le \frac{16e \norm{k}_{L^\infty(\R^d)}  T^2 e^{CT}} {\delta}. 
\end{equation}
\end{proof}

An application of the Csiszar--Kullback--Pinsker inequality and the sub-additivity property proves the following \(L^1\) estimate. 
\begin{corollary}
In the setting of Theorem~\ref{theorem: main_estimate_smooth_setting} we obtain 
\begin{equation*}
\E \bigg( \sup\limits_{0 \le t \le T} \norm{\rho^{r,N}_t-\rho_t^{ \otimes r}}_{L^1(\R^{dr})}^2 \bigg) \le \frac{C}{N}. 
\end{equation*}

\end{corollary}

\begin{remark}
We want to end this section with a small comment how to make the calculations completely rigours. 
Basically we need to avoid the singularity of the \(\log\) function. We accomplish that by replacing the function \(x\log(y)\) with \((x+\epsilon)\log(y+\epsilon)\). This guarantees the application of the It\^{o}'s formula and the vanishing of the stochastic integrals in all calculations. Obviously we can not longer integrate over the whole space \(\R^{dN}\). Therefore a multiplication with a suitable cut-off function depending on some parameter \(\tilde{\epsilon}\) and than integrating over \(\R^{dN}\) does the job. Finally we apply all estimates to the approximating system. Since all estimates will be uniform in the parameters \(\epsilon,\tilde{\epsilon}\) everything will be well-defined and we can take the limit by connecting \(\tilde{\epsilon}\) with \(\epsilon\) such that all appearing approximation terms vanish. Normally \(\epsilon\) needs to vanish much faster than \(\tilde{\epsilon}\), i.e. \(\epsilon = \mathcal{O}( \tilde{\epsilon}^L)\) for big enough \(L>0\) as \(\tilde{\epsilon} \to 0\). 
\end{remark}

\section{Stability for the stochastic PDE's} \label{sec: stability_spde}
The goal of this section is to lower the smoothness assumptions made in Section~\ref{sec: relative_entropy_smooth_setting} and obtain the estimate in Theorem~\ref{theorem: main_estimate_smooth_setting} in the case \(k \in L^2(\R^d) \cap L^\infty(\R^d) \). Our strategy consists of mollifying the coefficients, applying Theorem~\ref{theorem: main_estimate_smooth_setting} and then use the almost everywhere convergence along a subsequence to conclude the relative entropy estimate. 
Hence, let us consider a non-negative smooth function \(J^1: \R^d \to \R\) with compact support in the unit ball and mass one. We replace the coefficients \(\sigma, \nu, k\) by its mollified versions. More precisely, we define for each index \(\alpha , \beta, \tilde{l}\) and \(\epsilon > 0\) the functions 
\begin{align*}
[\sigma_t(\cdot)\sigma_t^{\mathrm{T}}(\cdot)]_{(\alpha,\beta)} * J^\epsilon(x) &= \int_{\R^d} [\sigma_t(y)\sigma_t^{\mathrm{T}}(y)]_{(\alpha,\beta)} J^\epsilon(x-y) \Id y, \\
\nu_t^{\beta,\tilde{l}} * J^\epsilon(x) &= \int_{\R^d} \nu_t^{\beta,\tilde{l}}(y) J^\epsilon(x-y) \Id y,\\ 
k * J^\epsilon(x) &= \int_{\R^d} k(y)J^\epsilon(x-y) \Id y ,\\
\rho_0 * J^\epsilon(x) &= \int_{\R^d} \rho_0(y)J^\epsilon(x-y) \Id y,
\end{align*}
where \(J^\epsilon(x) = \frac{1}{\epsilon^d} J^1\big( \frac{x}{\epsilon} \big) \). 
With abuse of notation, let us denote the mollified coefficients by \(k^\epsilon,[\sigma_t(x)\sigma_t^{\mathrm{T}}(x)]_{(\alpha,\beta)}^\epsilon\), \([\nu_t(x)\nu_t^{\mathrm{T}}(x)]_{(\alpha,\beta)} ^\epsilon\) and the mollified initial data by \(\rho^\epsilon_0\). Notice, that the mollified functions also satisfy the regularity estimates in Assumption~\ref{ass: diffusion_coef} uniformly in \(\epsilon\). 
Additionally, by the properties of mollifiers we have \([\sigma_t(x)\sigma_t^{\mathrm{T}}(x)]_{(\alpha,\beta)}^\epsilon\), \([\nu_t(x)\nu_t^{\mathrm{T}}(x)]_{(\alpha,\beta)} ^\epsilon\) converging to \([\sigma_t(x)\sigma_t^{\mathrm{T}}(x)]_{(\alpha,\beta)}\), \([\nu_t(x)\nu_t^{\mathrm{T}}(x)]_{(\alpha,\beta)}\) uniformly on compact sets of \(\R^d\) and \(k^\epsilon,\rho^\epsilon\) converge in the \(L^2\)-norm towards \(k,\rho_0\). 

Let \(\rho^{N,\epsilon}\), \(\rho^{\epsilon}\) be the solution to~\eqref{eq: liouville_SPDE},~\eqref{eq: d_dim_spde} with the mollified coefficients. 
In the next step we demonstrate that the mollified solutions \(\rho^{N,\epsilon}\), \(\rho^{\epsilon}\) converge to \(\rho^{N}\), \(\rho\). 

\begin{lemma} \label{lemma: conv_moll_N_spde}
Fix \(N\in \N\) and let \(k\in L^2(\R^d) \cap L^\infty(\R^d) \). Then we have the following convergence between \(\rho^{N,\epsilon}\) and \(\rho^N\), 
\begin{equation*}
\lim\limits_{\epsilon \to 0 } \norm{\rho^{N, \epsilon}-\rho^N}_{L^2_{\cF^W}([0,T];H^1(\R^{d}))} = 0 .
\end{equation*}
\end{lemma}
\begin{proof}
We need to verify the stability assumptions of~\cite[Theorem 5.7]{Krylov1999AnAA}. 
Since the coefficients \(\sigma,\nu\) are continuous uniform in time, the mollified versions convergence in the supremums norm. Moreover, by the properties of mollification 
\begin{equation}\label{eq: mollified_initial_value_convergence}
    \lim\limits_{\epsilon \to 0}\norm{\rho_0^\epsilon -\rho_0}_{L^2(\R^d)} = 0 .
    \end{equation}
    Hence, we only need to take care of the drift term. 
Define 
\begin{equation*}
f(x,u)  = \sum\limits_{i=1}^N  \nabla_{x_i} \cdot \bigg( \frac{1}{N} \sum\limits_{j=1}^N k(x_i-x_j) u\bigg), \quad f^\epsilon(x,u) = \sum\limits_{i=1}^N  \nabla_{x_i} \cdot \bigg( \frac{1}{N} \sum\limits_{j=1}^N k^\epsilon(x_i-x_j) u \bigg)
\end{equation*}
for \(u\in L^2_{\cF^W}([0,T];H^{1}(\R^{d}))\) and observe that 
\begin{equation*}
\norm{f^\epsilon-f}_{L^2_{\cF^W}([0,T];H^{-1}(\R^{d}))} 
\le \frac{1}{N} \sum\limits_{i,j=1}^N \norm{(k^\epsilon(x_i-x_j) -k(x_i-x_j)) u}_{L^2_{\cF^W}([0,T];L^2(\R^{d}))}.
\end{equation*}
By the properties of mollification we can extract a subsequence, which we do not rename, such that \(\lim\limits_{\epsilon \to 0} k^\epsilon = k  \;  \text{a.e.}\) and \[|(k^\epsilon(x_i-x_j) -k(x_i-x_j)) u|^2 \le 2 \norm{k}_{L^\infty(\R^d)}^2 |u|^2.\] Hence, by the dominated convergence theorem we obtain 
\begin{equation*}
\lim\limits_{\epsilon \to 0} \norm{f^\epsilon-f}_{L^2_{\cF^W}([0,T];H^{-1}(\R^{d}))}  = 0
\end{equation*}
and we can apply~\cite[Theorem 5.7.]{Krylov1999AnAA}. 
\end{proof}

\begin{lemma} \label{lemma: convergence_moll_chaotic_spde}
Let \(k\in L^2(\R^d) \cap L^\infty(\R^d) \). Then, there exists a sequence of smooth solutions we have the following convergence between \(\rho^{\epsilon}\) and \(\rho\), 
\begin{equation*}
\lim\limits_{\epsilon \to 0 } \norm{\rho^{ \epsilon}-\rho}_{L^2_{\cF^W}([0,T];H^1(\R^{d}))} = 0 .
\end{equation*}
\end{lemma}
\begin{proof}
The proof is based on the evolution of the \(L^2\)-norm similar to Proposition~\ref{prop: existence_d_spde}. 
By equation~\eqref{eq: picard_iteration_conv} we know that the Picard iteration \((\rho^n, n \in \N)\) convergence. Hence, it is sufficient to find a smooth approximation for \(\rho^{n,\epsilon}\) for each fixed \(n \in \N\). Consequently, let \(\rho^{n,\epsilon}\) be the solution of the Picard iteration with mollified coefficients
\begin{align*}
\Id \rho^{n,\epsilon}_t
=& \;  \nabla \cdot (( k^{\epsilon}*\rho^{n-1,\epsilon}_t) \rho^{n,\epsilon}_t ))  \Id t 
- \nabla \cdot (\nu_t^\epsilon  \rho^{n,\epsilon} \Id W_t)  \\
&\; + \frac{1}{2}  \sum\limits_{\alpha,\beta=1}^d \partial_{x_{i,\alpha}}\partial_{x_{i,\beta}} \bigg( ( [\sigma_t \sigma_t^{\mathrm{T}}]_{(\alpha,\beta)}^\epsilon + [\nu_t^\epsilon (\nu_t^\epsilon)^{\mathrm{T}}]_{(\alpha,\beta)} ) \rho^{n,\epsilon}_t \bigg)  \Id t. 
\end{align*}

Then, 
\begin{align} \label{eq: limiting_aux1}
&\; \norm{\rho^{ \epsilon}-\rho}_{L^2_{\cF^W}([0,T];L^2(\R^{d}))}  \nonumber  \\
\le &\;  \norm{\rho^{n}-\rho}_{L^2_{\cF^W}([0,T];L^2(\R^{d}))} 
+ \norm{\rho^{ \epsilon,n}-\rho^{\epsilon}}_{L^2_{\cF^W}([0,T];L^2(\R^{d}))} 
+
\norm{\rho^{ \epsilon,n}-\rho^n}_{L^2_{\cF^W}([0,T];L^2(\R^{d}))}. 
\end{align}
The first term vanishes in the limit for \(n \to \infty\) by~\eqref{eq: picard_iteration_conv} and the second by similar arguments.
Notice that \(\rho^{\epsilon,n}-\rho^n\) solves the following SPDE 
\begin{align*}
&\; \Id (\rho^{n,\epsilon}_t-\rho^n_t) \\
=& \;  \nabla \cdot (( k^{\epsilon}*\rho^{n-1,\epsilon}_t) \rho^{n,\epsilon}_t - (k*\rho^{n-1}_t)\rho^{n}_t  ) \Id t 
- \nabla \cdot (((\nu_t^\epsilon -\nu_t)  \rho^{n} +\nu_t^{\epsilon} (  \rho^{n,\epsilon}-\rho^n_t) ) \Id W_t)  \\
& + \frac{1}{2}  \sum\limits_{\alpha,\beta=1}^d \partial_{x_{i,\alpha}}\partial_{x_{i,\beta}} \bigg( ( [\sigma_t \sigma_t^{\mathrm{T}}]_{(\alpha,\beta)}^\epsilon + [\nu_t^\epsilon (\nu_t^\epsilon)^{\mathrm{T}}]_{(\alpha,\beta)}  - [\sigma_t \sigma_t^{\mathrm{T}}]_{(\alpha,\beta)} + [\nu_t \nu_t^{\mathrm{T}}]_{(\alpha,\beta)} ) 
 \rho^{n}_t  \\
 &+ [\sigma_t \sigma_t^{\mathrm{T}}]_{(\alpha,\beta)}^\epsilon + [\nu_t^\epsilon (\nu_t^\epsilon)^{\mathrm{T}}]_{(\alpha,\beta)}  ( \rho^{n,\epsilon}_t -\rho^n_t) \bigg) \Id t,
\end{align*}
wherer the coeffients \(\sigma^\epsilon,\nu^\epsilon,\sigma,\nu\) all satisfy Assumptions~\ref{ass: diffusion_coef}. Indeed, the mollification preserves the uniform bounds and the divergence free property. By the non-negativity of the mollifier, the ellepticity  
\begin{equation*}
\sum\limits_{\alpha,\beta=1}^d      [\sigma_s \sigma_s^{\mathrm{T}}(x)]_{(\alpha,\beta)}^\epsilon    \lambda_{\alpha} \lambda_{\beta}
= \int_{\R^d} J^\epsilon(x-y) \sum\limits_{\alpha,\beta=1}^d  [\sigma_s \sigma_s^{\mathrm{T}}(x)]_{(\alpha,\beta)}\lambda_{\alpha} \lambda_{\beta} \Id y
\ge \delta |\lambda|^2   
\end{equation*}
also holds. 
Utilizing the linearity we can perform the same steps as in Proposition~\ref{prop: existence_d_spde} to obtain 
\begin{align*}
&\;  \norm{\rho^{ \epsilon,n}_t-\rho^n_t}_{L^2(\R^{d})}^2 
-\norm{\rho^{ \epsilon,n}_0-\rho^n_0}_{L^2(\R^{d})}^2  \\
\le &\;  -2 \int\limits_0^t \int_{\R^d}  (( k^\epsilon*\rho^{n-1,\epsilon}_s) \rho^{n,\epsilon}_s - (k*\rho^{n-1}_s) \rho^{n}_s )  \cdot \nabla  (\rho^{n,\epsilon}_s - \rho^{n}_s) \Id x \Id s  \\
& - \sum\limits_{\alpha,\beta=1}^d  \int\limits_0^t \int_{\R^d}  [\sigma_s \sigma_s^{\mathrm{T}}]_{(\alpha,\beta)}^\epsilon \partial_{x_{i,\beta}}  (\rho^{n,\epsilon}_s - \rho^{n}_s)  \partial_{x_{i,\alpha}} (\rho^{n,\epsilon}_s - \rho^{n}_s)  \Id x \Id s    \\
&  + \frac{\delta}{2} \int\limits_0^t  \norm{\nabla(\rho^{n,\epsilon}_s - \rho^{n}_s)}_{L^2(\R^d)}^2  \Id s +  C(d,\delta )\int\limits_0^t\norm{\rho^{n,\epsilon}_s - \rho^{n}_s}_{L^2(\R^d)}^2 \Id s  \\
&+  \sum\limits_{\alpha,\beta=1}^d   \int\limits_0^t \int_{\R^d}  ( [\sigma_s \sigma_s^{\mathrm{T}}]_{(\alpha,\beta)}^\epsilon
+ [\nu_s^\epsilon (\nu_s^\epsilon)^{\mathrm{T}}]_{(\alpha,\beta)}  - [\sigma_s \sigma_s^{\mathrm{T}}]_{(\alpha,\beta)} + [\nu_s \nu_s^{\mathrm{T}}]_{(\alpha,\beta)} ) 
 \rho^{n}_s \Id x \Id s  \\
 &+  \sum\limits_{\hat{l}=1}^{\tilde{m}}\int\limits_0^t \int_{\R^d} \bigg|\sum\limits_{\beta=1}^d \partial_{x_{i,\beta}} (\nu_s^{\beta,\hat{l},\epsilon}-\nu_s^{\beta,\hat{l}}) \rho^{n}_s \bigg|^2 \Id x\Id s  \\
 &+ 2 \sum\limits_{\hat{l}=1}^{\tilde{m}} \sum\limits_{\beta=1}^d  \int\limits_0^t  \int_{\R^d}  \rho^{n}_s(\rho^{n,\epsilon}_s-\rho_s^{n}) \partial_{x_{i,\beta}} (\nu_s^{\beta,\hat{l},\epsilon}-\nu_s^{\beta,\hat{l}} ) \Id x  \Id W_s^{\hat{l}} .
\end{align*}
At the moment we can ignore the last term, since it will vanish after taking the expectation. For the penultimate term we use Lemma~\ref{lem: moment_estimate_spde}. Let us start with the case \(d\ge 3\) and denote by \(2^* = 2d/(d-2)\) the Sobolev exponent. Then for all \(R \in \N\) we find 
\begin{align*}
&\; \E \bigg( \sum\limits_{\hat{l}=1}^{\tilde{m}}\int\limits_0^t \int_{\R^d} \bigg|\sum\limits_{\beta=1}^d \partial_{x_{i,\beta}} (\nu_s^{\beta,\hat{l},\epsilon}-\nu_s^{\beta,\hat{l}}) \rho^{n}_s \bigg|^2 \Id x\Id s \bigg) \\
\le &\;   \sum\limits_{\hat{l}=1}^{\tilde{m}}  \sum\limits_{\beta=1}^d \E \bigg(  \int\limits_0^t \int_{B(0,R)}   |(\nu_s^{\beta,\hat{l},\epsilon}-\nu_s^{\beta,\hat{l}}) \rho^{n}_s|^2 + \int_{B(0,R)^{\mathrm{c}}} |  (\nu_s^{\beta,\hat{l},\epsilon}-\nu_s^{\beta,\hat{l}}) \rho^{n}_s|^2 \Id x \Id s \bigg)\\
\le & \;  \sum\limits_{\hat{l}=1}^{\tilde{m}}  \sum\limits_{\beta=1}^d T C \big(\norm{\rho_0}_{L^2(\R^d)} \big) \norm{\nu_s^{\beta,\hat{l},\epsilon}-\nu_s^{\beta,\hat{l}}}_{C^1(B(0,R))} + C \E \bigg( \int\limits_0^t \norm{\rho_s^{n}}_{L^1(B(0,R)^{\mathrm{c}})}^{\frac{4}{d+2}} \norm{\rho_s^{n}}_{L^{2^*}(\R^d)}^{\frac{2d}{d+2}} \Id s  \bigg)  \\
\le & \;  \sum\limits_{\hat{l}=1}^{\tilde{m}}  \sum\limits_{\beta=1}^d  T C \big(\norm{\rho_0}_{L^2(\R^d)} \big) \norm{\nu_s^{\beta,\hat{l},\epsilon}-\nu_s^{\beta,\hat{l}}}_{C^1(B(0,R))}  \\
&+ \frac{C}{R^{\frac{8}{d+2}}} \E \bigg( \int\limits_0^t \norm{\rho_s^{n} |\cdot|^2}_{L^1(\R^d)}^{\frac{4}{d+2}} \norm{\rho_s^{n}}_{H^{1}(\R^d)}^{\frac{2d}{d+2}} \Id s  \bigg) \\
\le & \;  \sum\limits_{\hat{l}=1}^{\tilde{m}}  \sum\limits_{\beta=1}^d  T C \big(\norm{\rho_0}_{L^2(\R^d)} \big) \norm{\nu_s^{\beta,\hat{l},\epsilon}-\nu_s^{\beta,\hat{l}}}_{C^1(B(0,R))} \\
&+ \frac{CT^{\frac{2}{d+2}}}{R^{\frac{8}{d+2}}}  \int\limits_0^T \E \bigg(  \norm{\rho_s^{n} |\cdot|^2}_{L^1(\R^d)}^{2} \bigg)^{\frac{2}{d+2}} \norm{\rho_s^{n}}_{{L^2_{\cF^W}([0,T];H^{1}(\R^{d}))}}^{\frac{2d}{d+2}} \\
\le & \;  \sum\limits_{\hat{l}=1}^{\tilde{m}}  \sum\limits_{\beta=1}^d  T C \big(\norm{\rho_0}_{L^2(\R^d)} \big) \norm{\nu_s^{\beta,\hat{l},\epsilon}-\nu_s^{\beta,\hat{l}}}_{C^1(B(0,R))} + C(\rho_0,T) R^{-\frac{8}{d+2}}, 
\end{align*}
where we used the interpolation inequality for \(L^p\) spaces and Assumption~\ref{ass: diffusion_coef}~\eqref{item: regularity_common_noise} in the third step, the Sobolev embedding in the fourth step and finally inequality~\eqref{eq: sode_h1_picard_estimate} and Lemma~\ref{lem: moment_estimate_spde} in the last step. 
In the case \(d=2\) we obtain a similar estimate by using the \(L^q\)-bound on \(\rho^n\)~\cite[Theorem 12.33]{LeoniGiovanni2017Afci} for all \(q \in [2,, \infty) \). 
Utilizing the same split of domains and applying Lemma~\ref{lem: moment_estimate_spde} we obtain 
\begin{align*}
  & \; \sum\limits_{\alpha,\beta=1}^d     \E \bigg(   \int\limits_0^t \int_{\R^d}  ( [\sigma_s \sigma_s^{\mathrm{T}}]_{(\alpha,\beta)}^\epsilon
+ [\nu_s^\epsilon (\nu_s^\epsilon)^{\mathrm{T}}]_{(\alpha,\beta)}  - [\sigma_s \sigma_s^{\mathrm{T}}]_{(\alpha,\beta)} + [\nu_s \nu_s^{\mathrm{T}}]_{(\alpha,\beta)} ) 
 \rho^{n}_s \Id x \Id s  \bigg)  \\
\le & \; T  \sum\limits_{\alpha,\beta=1}^d \bigg( \sum\limits_{\hat{l}=1}^{\tilde{m}}  \norm{\nu_s^{\beta,\hat{l},\epsilon}-\nu_s^{\beta,\hat{l}}}_{C^1(B(0,R))} +  
\norm{[\sigma_s \sigma_s^{\mathrm{T}}]_{(\alpha,\beta)}^\epsilon- [\sigma_s \sigma_s^{\mathrm{T}}]_{(\alpha,\beta)} }_{C^1(B(0,R))} \bigg )  \\
&+  C(\rho_0) T^{\frac{1}{2}} R^{-2}. 
\end{align*}
Recall that \([\sigma_s \sigma_s^{\mathrm{T}}]_{(\alpha,\beta)}^\epsilon \) is still elliptic and therefore 
\begin{equation*}
     \sum\limits_{\alpha,\beta=1}^d  \int\limits_0^t \int_{\R^d}  [\sigma_s \sigma_s^{\mathrm{T}}]_{(\alpha,\beta)}^\epsilon \partial_{x_{i,\beta}}  (\rho^{n,\epsilon}_s - \rho^{n}_s)  \partial_{x_{i,\alpha}} (\rho^{n,\epsilon}_s - \rho^{n}_s)  \Id x \Id s 
     \ge \delta  \int\limits_0^t \int_{\R^d} |\nabla ( \rho^{n,\epsilon}_s - \rho^{n}_s ) |^2 \Id x \Id s . 
\end{equation*}
Combing the last three inequalities we obtain
\begin{align*}
    &\;  \E\bigg( \norm{\rho^{ \epsilon,n}_t-\rho^n_t}_{L^2(\R^{d})}^2 \bigg) \\
    \le &\;    \E\bigg( -\frac{\delta}{2} \; \int\limits_0^t \int_{\R^d} |\nabla ( \rho^{n,\epsilon}_s - \rho^{n}_s ) |^2 \Id x \Id s +  C(\rho_0) T^{\frac{1}{2}} R^{-1} \\
    &  -2 \int\limits_0^t \int_{\R^d}  (( k^\epsilon*\rho^{n-1,\epsilon}_s) \rho^{n,\epsilon}_s - (k*\rho^{n-1}_s) \rho^{n}_s )  \cdot \nabla   (\rho^{n,\epsilon}_s - \rho^{n}_s) \Id x \Id s \\
    &+  \sum\limits_{\hat{l}=1}^{\tilde{m}}  \sum\limits_{\beta=1}^d  T C \big(\norm{\rho_0}_{L^2(\R^d)} \big) \norm{\nu_s^{\beta,\hat{l},\epsilon}-\nu_s^{\beta,\hat{l}}}_{C^1(B(0,R))} + C(\rho_0,T) R^{-\frac{4}{d+2}} \\
    &+  T  \sum\limits_{\alpha,\beta=1}^d \bigg( \sum\limits_{\hat{l}=1}^{\tilde{m}}  \norm{\nu_s^{\beta,\hat{l},\epsilon}-\nu_s^{\beta,\hat{l}}}_{C^1(B(0,R))} +  
\norm{[\sigma_s \sigma_s^{\mathrm{T}}]_{(\alpha,\beta)}^\epsilon- [\sigma_s \sigma_s^{\mathrm{T}}]_{(\alpha,\beta)} }_{C^1(B(0,R))} \bigg ) \bigg).
\end{align*}
We notice that by the mollification properties the last terms will vanish.
The only difficulty remaining is the drift term. Young's inequality implies 
\begin{align*}
     &\;  -2 \int\limits_0^t \int_{\R^d} (( k^\epsilon*\rho^{n-1,\epsilon}_s) \rho^{n,\epsilon}_s - (k*\rho^{n-1}_s) \rho^{n}_s )  \cdot \nabla   (\rho^{n,\epsilon}_s - \rho^{n}_s) \Id x \Id s \\ 
     \le &\; \int\limits_0^t \int_{\R^d}  \frac{1}{2\delta} |( k^\epsilon*\rho^{n-1,\epsilon}_s) \rho^{n,\epsilon}_s - (k*(\rho^{n-1}_s)\rho^{n}_s ) |^2
     + \frac{\delta}{2} |\nabla ( \rho^{n,\epsilon}_s - \rho^{n}_s ) |^2  \Id x \Id s.
\end{align*}
Notice, that the last term can be absorbed by the diffusion. For the first term we obtain
\begin{align*}
&\; \frac{1}{2\delta}    \int\limits_0^t \int_{\R^d}   |( k^\epsilon*\rho^{n-1,\epsilon}_s) \rho^{n,\epsilon}_s - (k*(\rho^{n-1}_s)\rho^{n}_s ) |^2 \Id x \Id s  \\
\le &\;  \frac{2}{\delta}    \int\limits_0^t \int_{\R^d}   |(k^\epsilon-k)*\rho^{n-1,\epsilon}_s) \rho^{n,\epsilon}_s|^2 +( k*(\rho^{n-1,\epsilon}_s-\rho^{n-1}_s)) \rho^{n,\epsilon}_s +  k*\rho^{n-1}_s(\rho^{n,\epsilon}_s- \rho^{n}_s ) |^2 \Id x \Id s \\
\le &\; \frac{2C(\rho_0)}{\delta} \bigg( T\norm{k^\epsilon-k}_{L^2(\R^d)}^2 + \int\limits_0^t \norm{\rho^{n-1,\epsilon}_s-\rho^{n-1}_s}_{L^2(\R^d)}^2 \Id s +  \int\limits_0^t \norm{\rho^{n,\epsilon}_s-\rho^{n}_s}_{L^2(\R^d)}^2 \Id s \bigg) ,
\end{align*}
where we used Lemma~\ref{lem: priori_l2} and Young's inequality for convolutions. Subtituting this inequality and applying Gronwall's lemma we find 
\begin{align}
    \E\bigg( \norm{\rho^{ \epsilon,n}_t-\rho^n_t}_{L^2(\R^{d})}^2 \bigg) 
    \le \bigg(A(\nu,\sigma,R,\epsilon)  + \int\limits_0^t \E\big( \norm{\rho^{n-1,\epsilon}_s-\rho^{n-1}_s}_{L^2(\R^d)}^2 \big) \Id s \bigg) e^{CT}, 
\end{align}
where
\begin{align*}
    A(\nu,\sigma,R,\epsilon) :=  &\; \sum\limits_{\hat{l}=1}^{\tilde{m}}  \sum\limits_{\beta=1}^d  T C \big(\norm{\rho_0}_{L^2(\R^d)} \big) \norm{\nu_s^{\beta,\hat{l},\epsilon}-\nu_s^{\beta,\hat{l}}}_{C^1(B(0,R))} + C(\rho_0,T) R^{-\frac{8}{d+2}} \\
    +&T  \sum\limits_{\alpha,\beta=1}^d \bigg( \sum\limits_{\hat{l}=1}^{\tilde{m}}  \norm{\nu_s^{\beta,\hat{l},\epsilon}-\nu_s^{\beta,\hat{l}}}_{C^1(B(0,R))} +  
\norm{[\sigma_s \sigma_s^{\mathrm{T}}]_{(\alpha,\beta)}^\epsilon- [\sigma_s \sigma_s^{\mathrm{T}}]_{(\alpha,\beta)} }_{C^1(B(0,R))} \bigg ) \\
&+  C(\rho_0) T^{\frac{1}{2}} R^{-2} +\frac{2 C(\rho_0) }{\delta} T \norm{k-k^\epsilon}_{L^2(\R^d)}^2 + \norm{\rho^{ \epsilon,n}_0-\rho^n_0}_{L^2(\R^{d})}^2
\end{align*}
Applying the inequality \(n \) times we arrive at 
\begin{equation*}
   \E\bigg( \norm{\rho^{ \epsilon,n}_t-\rho^n_t}_{L^2(\R^{d})}^2 \bigg) 
   \le A(\nu,\sigma,R,\epsilon)e^{CT} \sum\limits_{j=0}^{n-1}  \frac{T^je^{CTj}}{j!} 
   + \frac{ e^{CTn}}{n!}\norm{\rho_0^{\epsilon}-\rho_0}_{L^2(\R^d)}^2,
\end{equation*}
which implies 
\begin{equation*}
  \sup\limits_{n \in \N}   \E\bigg( \norm{\rho^{ \epsilon,n}_t-\rho^n_t}_{L^2(\R^{d})}^2 \bigg) 
  \le  A(\nu,\sigma,R,\epsilon) e^{CT}e^{Te^{CT}} + C \norm{\rho_0^{\epsilon}-\rho_0}_{L^2(\R^d)}^2
\end{equation*}
and finally the convergence
\begin{align*}
\limsup\limits_{\epsilon \to 0}  \sup\limits_{n \in \N} \norm{\rho^{ \epsilon,n}-\rho^n}_{L^2_{\cF^W}([0,T];L^2(\R^{d}))} =0,
\end{align*}
by taking first \(\epsilon \to 0\) in combination with~\eqref{eq: mollified_initial_value_convergence} and the properties of mollifiers, and then \(R \to \infty\). 
Together with inequality~\eqref{eq: limiting_aux1} and the subsequent comment, this implies 
\begin{equation*}
  \limsup\limits_{\epsilon \to 0}  \norm{\rho^{ \epsilon}-\rho}_{L^2_{\cF^W}([0,T];L^2(\R^{d}))} =0.   
\end{equation*}
Applying~\cite[Theorem 5.1]{Krylov1999AnAA}, the \(L^2\)-bound~\eqref{eq: uniform_L2_bound} and \(k\in L^2(\R^d)\) we arrive at 
\begin{align*}
    \limsup\limits_{\epsilon \to 0}  \norm{\rho^{ \epsilon}-\rho}_{L^2_{\cF^W}([0,T];H^1(\R^{d}))} 
    &\le \limsup\limits_{\epsilon \to 0}  \norm{k*(\rho^{ \epsilon}-\rho) \rho }_{L^2_{\cF^W}([0,T];L^2(\R^{d}))}  \\
    &\le C  \limsup\limits_{\epsilon \to 0}  \norm{k*(\rho^{ \epsilon}-\rho) }_{L^2_{\cF^W}([0,T];L^\infty(\R^{d}))} \\
    &\le C \norm{k}_{L^2(\R^d)}  \limsup\limits_{\epsilon \to 0}  \norm{\rho^{ \epsilon}-\rho  }_{L^2_{\cF^W}([0,T];L^2(\R^{d}))} \\
    &= 0 . 
\end{align*}

\end{proof}

Finally, we can present the analogous result to~\cite[Theorem 1]{JabinWang2018} for bounded kernels. 

\begin{theorem} \label{theorem: main_estimate}
Let \(k \in L^2(\R^d) \cap L^\infty(\R^d)\), \(\rho^N\) be a solution to the Liouville equation~\eqref{eq: liouville_SPDE} and \(\rho^{\otimes N} \) be a solution to the chaotic SPDE~\eqref{eq: chaotic_spde}. Then, we have the following relative entropy bound
\begin{equation}
\E \bigg(  \sup\limits_{0 \le t \le T }  \mathcal{H}(\rho_t^{N}  \vert \rho_t^{\otimes N}) \bigg) \le C 
\end{equation}
for some \(C>0\) depending on \(\norm{k}_{L^\infty(\R^d)}\). 
\end{theorem}
\begin{proof}
By Lemma~\ref{lemma: conv_moll_N_spde} and Lemma~\ref{lemma: convergence_moll_chaotic_spde} we know there exists a subsequence, which we do not rename such that \(\rho^{N,\epsilon}\) and \(\rho^{\otimes N, \epsilon} \) converge almost everywhere on \(\Omega \times [0,T] \times \R^{dN}\). Hence, by the lower  semicontinuity of the relative entropy, the estimate \(\norm{k^\epsilon}_{L^\infty(\R^{d})} \le C \norm{k}_{L^\infty(\R^d)} \) and Theorem~\ref{theorem: main_estimate_smooth_setting} our claim follows. 
\end{proof}
As in the smooth case, we can also use the Csiszar--Kullback--Pinsker inequality and sub-additivity property to prove \(L^1\)-convergence of every \(r\)-marginal. 
\begin{corollary}
In the setting of Theorem~\ref{theorem: main_estimate}, we obtain 
    \begin{equation*}
  \E \bigg(  \sup\limits_{0 \le t \le T}\norm{\rho^{r,N}_t-\rho_t^{\otimes r}}_{L^1(\R^{dr})}^2 \bigg) \le \frac{Cr}{N}
\end{equation*}
for every \(r \in \N\). 
\end{corollary}

\section{Conditional propagation of Chaos}

At the moment it is unclear whether the convergence results in Section~\ref{sec: stability_spde} implies conditional propagation of chaos in the sense of weak convergence of empirical measures. 
In this section we will demonstrate some useful comparison results between the classical propagation of chaos towards a deterministic measure and the conditional propagation of chaos in the setting of common noise.

We demonstrate the following conditional exchangability for the interacting particle system~\eqref{eq: particle_system}, which is based on~\cite[Lemma 23]{CoghiFlandoli2016} adjusted to our setting.
\begin{lemma} \label{lemma: exchangable_enviro}
    Let \(t \ge 0\) and \((X^{i}_t,i = 1,\ldots, N)\) be given by~\eqref{eq: particle_system}. Then for any permutation \(\vartheta \colon \{1,\ldots,N\} \mapsto \{1,\ldots,N \}\) the vector of random variables \((X^{i}_t,i = 1,\ldots, N)\) satisfies 
    \begin{equation*}
        \E(h(X_t^{1},X_t^2,\ldots,X_t^N) \vert \cF_t^W ) 
        =  \E(h(X_t^{\vartheta(1)},X_t^{\vartheta(2)},\ldots,X_t^{\vartheta(N)}) \vert \cF_t^W ) 
    \end{equation*}
    for every \(h \in C_b(\R^{dN})\). 
\end{lemma}
\begin{remark}
 In particular the condition~\eqref{eq: conditional_exchanagbility_condition} in Lemma~\ref{lemma: propagation_of_chaos_aux} is fulfilled with \(\cG = \cF_t^{W}\).    
\end{remark}

\begin{proof}
Consider the particle system~\eqref{eq: particle_system} without common noise. Then by~\cite{hao2022} this SDE has a strong solution. Thus, we know that the particle system~\eqref{eq: particle_system} must also have a strong solution. Additionally, by the exchangeability of the initial data, the Yamada--Watanabe theorem tells us that strong uniqueness implies uniqueness in law and therefore we obtain 
\begin{align} \label{eq: aux_exchangable}
\begin{split} 
    &\; \mathrm{Law} \big( ((X_t^{1},X_t^2,\ldots,X_t^N),(W_t^{i},i =1,\ldots, \tilde{m})) \big) \\
    = &\; \mathrm{Law} \big( ((X_t^{\vartheta(1)},X_t^{\vartheta(2)}, \ldots,X_t^{\vartheta(N)}),(W_t^{i},i =1,\ldots, \tilde{m})) \big) .
\end{split}
\end{align}

Now choose a cylinder set \( A \in \cF_t^{W}\), i.e. 
\begin{equation*}
    A= (W_{t_{1,1}}^{1})^{-1}(A_{1,1}) \cap \cdots \cap (W_{t_{1,r_1}}^{1})^{-1}(A_{1,r_1}) \cap (W_{t_{2,1}}^{2})^{-1}(A_{2,1})  \cap \cdots \cap (W_{t_{\tilde{m},r_{\Tilde{m}}}}^{\tilde m})^{-1}(A_{\tilde{m},r_{\tilde{m}}})
\end{equation*}
for \(r_1,\ldots,r_{\tilde{m}}  \in \N \), \(t_{i,\gamma} \in [0,t] \) for \(i =1, \ldots, \tilde{m}\), \(\gamma=1,\ldots,\max(r_1,\ldots, r_{\tilde{m}}) \).
These cylinder set are closed under intersections and generate the \(\sigma\)-algebra \(\cF_t^W\). 
For \(f \in C_b(\R^{dN})\) we find 
\begin{equation*}
\E ( \indicator{A}h(X_t^{1},X_t^2,\ldots,X_t^N) ) 
= \E(\indicator{A} h(X_t^{\vartheta(1)},X_t^{\vartheta(2)}, \ldots,X_t^{\vartheta(N)})) 
\end{equation*}
by the representation of \(A\) as the intersection of inverse images of \(W\) and the uniqueness of laws~\eqref{eq: aux_exchangable}. Consequently, the conditional expectations must coincide. 
\end{proof}

Let us partially transfer the convergence results by Sznitman~\cite[Proposition 2.2]{snitzman_propagation_of_chaos}
to the case of random limiting measure. 

\begin{lemma} \label{lemma: propagation_of_chaos_aux}
Let \((\Omega,\cF,\cP)\) be a probability space and let \(\cG \) be a sub-\(\sigma\)-algebra of \(\cF\). 
Let \(\mathbf{Z}^N = (Z^{1,N},\ldots,Z^{N,N}) \) be a \(\R^{dN}\)-valued exchangeable vector of measurable variables and \(f^{1,N} \in L^1_{\cG}(L^1(\R^d))\), \( f^{2,N} \in L^1_{\cG}(L^1(\R^{2d}) \) be the conditional densities given the \(\sigma\)-algebra \(\cG\) of the random variables \(Z^{1,N},(Z^{1,N},Z^{2,N})\), respectively. Additionaly, suppose there exists a measurable function \(g \in L^1_{\cG}(L^1(\R^d))  \cap L^\infty_{\cG}(L^1(\R^d))\) such that
\begin{equation*}
\lim\limits_{N \to \infty} \bigg( \norm{f^{1,N}-g}_{L^1_{\cG}(L^1(\R^d))} + \norm{f^{2,N}-g \otimes g}_{L^1_{\cG}(L^1(\R^{2d}))} \bigg)  =0  
\end{equation*}
and for any bounded continuous function \(h \in C_b(\R^d)\) the equaliy 
\begin{equation} \label{eq: conditional_exchanagbility_condition}
\E(h(Z^{1,N}) \vert \cG) = \E(h(Z^{i,N}) \vert \cG), \quad \mathrm{for} \; \mathrm{all} \; i \le N 
\end{equation}
holds. 
Then, the empirical measure converges to \(g\) in the sense of random measures equipped with the topology of weak convergence. More precisely, for any bounded continuous function \(\varphi \in \R^d\) we have 
\begin{equation*}
 \bigg\langle \frac{1}{N} \sum\limits_{i=1}^N  \delta_{Z^{i,N}}, \varphi \bigg \rangle {\xrightarrow[N \to \infty]{d}} \; \langle g,\varphi \rangle  , 
\end{equation*}
where the convergence is in the sense of distributions. 
\end{lemma}

\begin{proof}
Let \(\varphi\colon \R^d \mapsto \R \) be a bounded continuous function. We show \(L^2\)-converges, which then implies converges in distribution. 
Expanding the square we find 
\begin{align*}
&\; \E \bigg( \bigg | \bigg\langle \frac{1}{N} \sum\limits_{i=1}^N  \delta_{Z^{i,N}}, \varphi \bigg \rangle
- \langle g,\varphi \rangle \bigg|^2 \bigg ) \\ 
=&\;  \frac{1}{N^2} \sum\limits_{i,j=1}^N  \E ( \varphi(Z^{i,N}) \varphi(Z^{j,N}) )-2 \frac{1}{N} \sum\limits_{i=1}^N \E(\varphi(Z^{i,N})  \langle g,\varphi \rangle ) +  \E(  \langle g,\varphi \rangle^2 )  \\
=&\; \frac{1}{N} \E(\varphi(Z^{i,N}) ^2)
+ \frac{(N^2-N)}{N^2}  \E ( \varphi(Z^{1,N}) \varphi(Z^{2,N}) )
- 2 \E(\E(\varphi(Z^{1,N}) \vert \cG)   \langle g,\varphi \rangle)
+  \E(  \langle g,\varphi \rangle^2 ) . 
\end{align*}
Notice that \(\varphi\) is bounded and therefore the first term vanishes. The factor of the second term converges to one and for the expected value we obtain
\begin{align*}
|\E ( \varphi(Z^{1,N}) \varphi(Z^{2,N}) ) -  \E(\langle \varphi , g \rangle^2 )| 
&= |\E ( \varphi(Z^{1,N}) \varphi(Z^{2,N}) \vert \mathcal{G}) -  \E(\langle \varphi , g \rangle^2 )| \\
&= |\E(\langle \varphi \otimes \varphi, f^{2,N} \rangle ) -  \E(\langle \varphi \otimes \varphi, g \otimes g \rangle ) | \\
&\le \norm{\varphi}_{L^\infty(\R^d}^2 \norm{f^{2,N}-g \otimes g}_{L^1_{\cG}(L^1(\R^{2d}))}
  \\
&\to 
0 , \; \mathrm{as} \; N \to \infty.
\end{align*}
For the third term we find 
\begin{align*}
|\E(\E(\varphi(Z^{1,N}) \vert \cG)   \langle \varphi, g  \rangle) -  \E(\langle \varphi , g \rangle^2 )|
&=|\E(( \langle \varphi, f^{1,N} \rangle  -  \langle \varphi, g  \rangle)  \langle \varphi , g \rangle)|  \\
& \le \norm{ \varphi}_{L^\infty(\R^d)} \norm{ f^{1,N}- g  }_{L^1_{\cG}(L^1(\R^{d}))} \norm{ \langle \varphi , g \rangle}_{L^\infty_{\cG}}  \\
&\le  \norm{ \varphi}_{L^\infty(\R^d)}^2 \norm{g}_{L^\infty_{\cG}(L^1(\R^d))} \norm{ f^{1,N}- g  }_{L^1_{\cG}(L^1(\R^{d}))} \\
&\to 
0 , \; \mathrm{as} \; N \to \infty.
\end{align*}
Combining the two convergences proves the lemma. 
\end{proof}

Combining Lemma~\ref{lemma: propagation_of_chaos_aux}, Lemma~\ref{lemma: exchangable_enviro} and Theorem~\ref{theorem: main_estimate} we obtain conditional propagation of chaos.                  
\begin{corollary}
Let \(k \in L^2(\R^d) \cap L^\infty(\R^d)  \), then conditional propagation of chaos holds for the interacting particle system~\eqref{eq: particle_system} towards the random measure with density \(\rho\) given by the solution of the SPDE~\eqref{eq: d_dim_spde}. 
\end{corollary}

\appendix
\begin{center}
{Appendix}
\end{center}

\renewcommand{\thesection}{A}
\setcounter{theorem}{0}

\medskip

We demonstrate that the relative entropy on laws can be estimated by the relative entropy on the conditional laws. Consequently, demonstrating convergence in relative entropy on some conditional laws is always a stronger statement, which justifies the use of the conditional Liouville equation~\eqref{eq: liouville_SPDE}. 
\begin{lemma}
Let \(X,Y\) be two Banach space \((E,\mathcal{E})\)-valued random variables defined on some probability space \((\Omega, \cF, \P)\). Let \(\cG^{i}\) for \(i=1,2\) be sub \(\sigma\)-algebras of \(\cF\). Than we have the following inequality 
\begin{equation*}
\mathcal{H}(\L_{X}  \vert  \L_{Y}) 
\le  \E(  \mathcal{H}(\L_{X \vert \cG^1 } \vert  \L_{Y \vert \cG^2 } )).
\end{equation*}
Choosing \(\cG^1\) or \(\cG^2\) as the trivial \(\sigma\)-algebra we obtain 
\begin{equation*}
\mathcal{H}(\L_{X}  \vert  \L_{Y}) 
\le \min \bigg( \E( \mathcal{H}(\L_{X \vert \cG^1 } \vert  \L_{Y} ) ), \E( \mathcal{H}(\L_{X  } \vert  \L_{Y \vert \cG^2 } ) ) \bigg). 
\end{equation*}

\end{lemma}
\begin{proof}
Using the variational formula and Jensen inequality we obtain 
\begin{align*}
\mathcal{H}(\L_{X}  \vert  \L_{Y}) 
&= \sup\limits_{\psi \in \mathcal{B}_b(E)}
\bigg ( \E(\psi(X)) - \log \big(  \E(\exp(\psi(Y)))\big) \bigg)  \\
&=\sup\limits_{\psi \in \mathcal{B}_b(E)}
\bigg ( \E(\psi(X)) - \log \big(  \E(\E(\exp(\psi(Y)) \vert \cG^2 ) ) \big) \bigg)  \\
&\le \sup\limits_{\psi \in \mathcal{B}_b(E)}
\bigg ( \E(\E(\psi(X) \vert \cG^1) ) -  \E( \log ( \E(\exp(\psi(Y)) \vert \cG^2 )  ) ) \bigg) \\
&= \sup\limits_{\psi \in \mathcal{B}_b(E)}
\bigg ( \E \bigg( \E(\psi(X) \vert \cG^1) -   \log \big( \E(\exp(\psi(Y)) \vert \cG^2 )  \big) \bigg) \bigg)\\
&\le \E \bigg( \sup\limits_{\psi \in \mathcal{B}_b(E)}\bigg( \E(\psi(X) \vert \cG^1) -   \log \big( \E(\exp(\psi(Y)) \vert \cG^2 )  \big) \bigg) \bigg)\\
 &=  \E(  \mathcal{H}(\L_{X \vert \cG^1 } )   \vert  \L_{Y \vert \cG^2 } )). 
\end{align*}
\end{proof}

\bibliography{quellen}
\bibliographystyle{amsalpha}
\end{document}